\documentclass{article}
\usepackage[english]{babel}
\usepackage{graphicx}
\usepackage{slashed}
\usepackage{stmaryrd}
\usepackage{amsmath,amsfonts,amssymb}

\usepackage{ntheorem}
\usepackage{dsfont}
\usepackage{mathtools}
\usepackage{verbatim}
\usepackage{float}
\usepackage{mathabx}
\usepackage{accents}
\usepackage{mathrsfs}
\usepackage{enumerate}

\usepackage{hyperref}
\usepackage{tikz}
\usepackage{pgfplots}
\pgfplotsset{compat=1.14}
\usetikzlibrary{patterns}
\usetikzlibrary{positioning,arrows,arrows.meta}
\usepackage[left=2cm, right=2cm, bottom=2cm, top=2cm]{geometry}
\usepackage{titlesec}

\titleformat{\subsection}[runin]
  {\normalfont\bfseries}
  {\thesubsection}
  {1em}
  {}
  [.]

\titleformat{\subsubsection}[runin]
  {\normalfont\itshape}
  {\thesubsubsection}
  {1em}
  {}
  [.]

\usepackage{enumitem}
\theoremseparator{.}

\makeatletter
\renewcommand\tableofcontents{%
  \section*{\centering\scshape Contents}%
  \@starttoc{toc}%
}
\makeatother

\makeatletter

\renewcommand*\l@section[2]{%
  \ifnum \c@tocdepth >\z@
    \addpenalty\@secpenalty
    \addvspace{0pt}%
    \setlength\@tempdima{1.5em}%
    \begingroup
      \parindent \z@
      \rightskip \@pnumwidth
      \parfillskip -\@pnumwidth
      \leavevmode
      \advance\leftskip\@tempdima
      \hskip -\leftskip
      #1\nobreak\hfil
      \nobreak\hb@xt@\@pnumwidth{\hss #2}\par
    \endgroup
  \fi
}

\makeatother

\newtheorem{Th}{Theorem}[section]
\newtheorem{Def}[Th]{Definition}
\newtheorem{Rq}[Th]{Remark}
\newtheorem{Pro}[Th]{Proposition}
\newtheorem{Cor}[Th]{Corollary}
\newtheorem{Lem}[Th]{Lemma}
\newtheorem{Conj}[Th]{Conjecture}

\newcommand{\f}{f^{\mathrm{lin}}}
\newcommand{\A}{A_I^{\mathrm{lin}}}
\newcommand{\R}{\mathbb{R}}

\newcommand{\T}{\mathbf{T}}
\newcommand{\Pp}{\mathbb{P}}
\newcommand{\dr}{\mathrm{d}}

\newenvironment{proof}{\noindent\textit{Proof.~}}{\hfill$\square$\bigbreak} 

\title{Global dynamics of Vlasov--wave systems without the null condition under exponential momentum decay}

\author{L\'eo Bigorgne\footnote{Institut de Recherche Math\'ematique de Rennes (IRMAR) - UMR 6625, CNRS, Universit\'e de Rennes, F-35000 Rennes, France.
{\em E-mail address:} {\tt leo.bigorgne@univ-rennes.fr}} }

\date{}
\begin{document}

\maketitle
    
\begin{abstract}
We prove global existence and modified scattering for solutions to relativistic Vlasov--wave systems without null structure, arising from small distribution functions. When the homogeneity of the force field in the momentum variable is the same as in the Vlasov-Maxwell system, the scalar fields are allowed to be large and our method requires stretched exponential velocity decay. When it matches that of the Einstein-Vlasov system, the scalar fields are required to be small and the distribution function to have slightly super-exponential velocity decay, which in particular includes Maxwellian decay. Moreover, we show that the smallness assumption on the scalar fields cannot in general be removed, even for compactly supported initial data. 
\end{abstract}

  \tableofcontents

\section{Introduction}

The goal of this paper is to study the small data solutions to Vlasov-wave systems. More precisely, we will study systems of the form
\begin{equation}\label{VW}
\begin{cases}
\partial_t f+\widehat{v} \cdot \nabla_x f + \nabla \mathbf{A} \cdot \nabla_v f=0,\\
\Box A_I=\pmb{\rho}_I[f],\qquad I \in \llbracket  1, K \rrbracket , \tag{Vw--$\delta$}
\end{cases}
\end{equation}
where $\delta \in \{0,1\} $, $ K \in \mathbb{N}^*$ and
\begin{itemize}
\item $f \colon \R_t \times \R^3_x \times \R^3_v \to \R$ is the density distribution function of a large ensemble of particles of mass $m=1$. Although physically $f$ is nonnegative, this assumption plays no role in our analysis, and we therefore allow $f$ to take arbitrary real values.

\item $\widehat{v}=\dfrac{v}{\langle v\rangle}$, where $\langle v \rangle \coloneqq \sqrt{1+|v|^2}$, is the relativistic speed of a particle of momentum $v \in \R^3_v$.

\item The scalar fields $A_I \colon \R_t \times \R^3_x \to \R$ give rise to the interaction field $\nabla\mathbf{A}:\R_t \times \R^3_x \times\R^3_v \to \R^3$, which is a schematic notation denoting
\begin{equation}\label{defA}
\big[\nabla\mathbf{A} \big]^i(t,x,v) \coloneqq \sum_{1\leq I\le K} \sum_{0\leq \mu \leq 3} \langle v\rangle^\delta Q_I^{i,\mu}(\widehat{v})\partial_{x^\mu}A_I(t,x), \qquad i \in \llbracket 1,3 \rrbracket,
\end{equation}
where $Q_I^{i,\mu}\in \R_3[\widehat{v}]$ are polynomials in $(\widehat{v}^1,\widehat{v}^2,\widehat{v}^3)$ and $\partial_{x^0}=\partial_t$. We do not assume $\nabla_{x,v} \cdot (\widehat{v},\nabla \mathbf{A})=0$.

\item The source terms $\pmb{\rho}_I[f]$ are momentum averages of $f$,
\[ \pmb{\rho}_I[f] \coloneqq \int_{\R^3_v}\langle v\rangle^{\delta'} P_I(\widehat{v})f \dr v,
\qquad P_I\in\R_3[\widehat{v}], \quad \delta' \leq \delta.\]
\end{itemize}

Our motivation is to investigate to what extent properties of small data solutions known for relativistic Vlasov systems arising in mathematical physics persist in the more general framework \eqref{VW}. Previous works on physically motivated Vlasov systems have exploited their null structure to control particles with large momenta and thereby treat distribution functions with non-compact momentum support (see Section \ref{Subsecnullcompo}). It is not a priori clear whether such data can still be handled in the absence of a null condition in the Vlasov operator. We show that this is possible for sufficiently fast decaying data, including in particular Maxwellian profiles (see Theorem \ref{Th} for our main result). We briefly discuss three main examples of physically motivated Vlasov systems here.
\begin{itemize}
    \item \textit{The Vlasov-Maxwell system.} In that case, $\delta =0$, $K=4$, and the interaction field is the Lorentz force $\nabla \mathbf{A} = E+\widehat{v} \times B$. One way to write this system in the form $(\text{Vw-0})$ is to consider a potential $A$ satisfying the Lorenz gauge\footnote{Strictly speaking, the equivalence with Vlasov–Maxwell requires the initial data to satisfy both the Lorenz gauge condition and the Maxwell constraint equations (in particular, Gauss' law). These constraints are propagated by the evolution.},
    \[A= A_0\dr t+ A_1 \dr x^1 + A_2\dr x^2 +A_3 \dr x^3, \qquad \qquad \partial_t A_0=\sum_{1 \leq i \leq 3} \partial_{x^i} A_i. \]
    Then, we have $E_i=\partial_t A_i - \partial_{x^i}A_0$, $B= \nabla \times (A_1,A_2,A_3)$, $\Box A_0 = - \int_v f \dr v$ and $\Box A_i = \int_v \widehat{v}^i f \dr v$.
    \item \textit{The Vlasov--Nordström system.} This model describes the evolution of collisionless matter in Nordström's scalar theory of gravitation (see for instance \cite[Section~2]{Calogero03}). Although it cannot be written exactly in the form \eqref{VW}, it is closely related to the case $K=1$, $\delta=1$ and $\delta'=-1$.
    \item \textit{The Einstein--Vlasov system.} In wave coordinates, the equations reduce to a system of $K=10$ coupled quasilinear wave equations for the metric components $g_{\mu\nu}$ coupled to a Vlasov field. Although the Einstein-Vlasov system is much more complicated than \eqref{VW}, it motivates our study of the case $\delta=\delta'=1$.
\end{itemize}

\subsection{Overview of the main issues}\label{Subsecdiff}

Consider a smooth solution $(f^{\mathrm{lin}},A_I^{\mathrm{lin}})$ to the linearised system around $0$, 
\[  \partial_t \f + \widehat{v} \cdot \nabla_x \f =0, \qquad \qquad \Box \A = \pmb{\rho}_I [\f] . \]
Then, one can in particular show that there exist $\mathbf{Q}_0(v)$, a function proportional to $\int_x \f(0,x , \cdot) \dr x$, and $\mathbb{A}_I^\mu(v)$, which are functionals of $\mathbf{Q}_0$, such that
\[ \pmb{\rho}_I \big[\f \big] (t,x) = \frac{1}{\langle t \rangle^3} \mathbf{Q}_0 \Big( \frac{x}{t} \Big)\mathds{1}_{|x| < t}+O \big( t^{-4} \big), \qquad \qquad \partial_{x^\mu}\A (t,x) = \frac{1}{\langle t \rangle^2} \mathbb{A}_I^\mu \Big( \frac{x}{t} \Big) \mathds{1}_{|x| < t}+ A_I^{\mu,\mathrm{rem}}(t,x), \]
where the remainder term is lower order in the timelike directions $t \mapsto (t,t\widehat{v})$ but not in the null directions $t \mapsto (t,x+t\omega)$, $\omega \in \mathbb{S}^2$. More precisely,
\[ \big| A^{\mu,\mathrm{rem}}_I \big| (t,x) \lesssim \langle t+|x| \rangle^{-1} \langle t-|x| \rangle^{-2}.  \]
These two different asymptotic behaviors reflect the following properties.
\begin{itemize}
\item Vlasov fields exhibit stronger decay in the exterior region and near the light cone, namely for $|x|>t$ and $|x|\sim t$, than in the interior region. This can already be seen from the characteristics $t \mapsto (x+t\widehat{v},v)$ of the free transport operator $\partial_t+\widehat{v}  \cdot\nabla_x$, which correspond to the trajectories of freely moving massive particles. In particular, for $t \gg 1$, most of the matter is located inside the future light cone and becomes increasingly separated from its boundary $\{t=|x|\}$.

\item The energy carried by waves is radiated toward future null infinity, that is, along null rays $t \mapsto (t,x+t\omega)$, $\omega \in \mathbb{S}^2$. As a consequence, the quantity $r \A (r+u,r\omega)$ is not expected to decay in $r$, but may decay in $|u|$.
\end{itemize}
In addition, we have boundedness for some derivatives of the distribution function,
$\nabla_x \f$ remains bounded and
\[ \big| \nabla_x \f \big| (t,x,v) \lesssim 1, \qquad \qquad  \big| \langle v \rangle \nabla_v \f +t\nabla_x \f -t\widehat{v} \, \widehat{v} \cdot \nabla_x \f \big|(t,x,v) \lesssim 1. \]
Consequently, if the asymptotics of small data solutions are sufficiently governed by that of the linearised system, one expects the nonlinearity in \eqref{VW} to behave as
\[ \big| \nabla \mathbf{A} \cdot \nabla_v \f \big|(t,x,v) \lesssim \Big( \langle t \rangle^{-1}+\langle t-|x| \rangle^{-2} \Big) \langle v \rangle^{\delta -1} \cdot \frac{\langle v \rangle}{ \langle t \rangle} \big| \nabla_v f^{\mathrm{lin}} \big|(t,x,v). \]
There are two problems with the rate of decay of the first factor in the right hand side.
\begin{enumerate}[ label = (\roman*)]
    \item \label{it1} It does not decay near the light cone, that is for $t\sim |x|$. As a consequence, it is unclear whether the velocity characteristics remain bounded. It is well known that propagation of regularity and decay estimates for kinetic equations are closely tied to the control of large momentum. The same difficulty arises when commuting the equation in order to control derivatives of the distribution function. 
    
    \item \label{it2} It is not integrable in time, even in the region $t >2|x|$. Because of that, and as previously established for the Vlasov-Maxwell system, the small data solutions to \eqref{VW} exhibits a modified scattering dynamics. As we will prove, the characteristics $t \mapsto (\mathcal{X}_t,\mathcal{V}_t)$ associated with the Vlasov operator does not asymptote a linear trajectory $t \mapsto (x+t\widehat{v},v)$. Instead, we will show that there exists a function $\mathscr{C} \colon \R^3 \to \R^3$ and $(x_\infty,v_\infty)$ such that 
    \begin{equation}\label{eq:intromodscatt}
    \mathcal{V}_t \to v_\infty , \qquad \mathcal{X}_t - t \widehat{v}_\infty - \log (t) \mathscr{C} (v) \to x_\infty  \qquad \text{as $t \to + \infty$}. 
    \end{equation}
    For instance, in the case where $K=1$ and $\nabla \mathbf{A}=\nabla_x A_1$, we have using the notation $\mathbb{A}_I = (\mathbb{A}^1_I, \mathbb{A}^2_I,\mathbb{A}^3_I)$, 
    \[\mathscr{C}(v) = \frac{\widehat{v} }{\langle v \rangle} \widehat{v} \cdot \mathbb{A}_1(v)- \frac{1}{\langle v \rangle} \mathbb{A}_1(v). \]

   Because of this, controlling derivatives of the solutions in the proof of global existence gives rise to non-integrable error terms. 
\end{enumerate}

To deal with \ref{it1}, on may expect to exploit the non-resonant nature of the interaction. In the favorable case of a compactly supported distribution function, one should be able to show that there exists two constants $A, \, w >0$ such that\footnote{As shown below in Proposition \ref{ProInstaLambdaLarge}, however, such a property may fail in the case $\delta =1$ when the scalar fields are sufficiently large.}
\begin{equation*}
f(t,x,v)=0 \qquad \text{ for all $|v| \geq w$ or $|x| \geq A+\widehat{w}t$,} \quad \widehat{w} \coloneqq w/\langle w \rangle.
\end{equation*}
As a consequence, $t-|x| \sim t$ on the support of $f$. In the non-compactly supported case, the characteristics of the Vlasov and wave operators are no longer uniformly separated. We capture the good behavior of $f$ near the light cone by exploiting the weight functions $\langle v \rangle$ and $\langle x -t \widehat{v} \rangle $, which are preserved along the linear flow, and the inequality
\[  \langle t+|x| \rangle \lesssim \langle t-|x| \rangle \, \langle x - t \widehat{v} \rangle \, \langle v \rangle^2 . \]
However, for the strongest weighted norms, this inequality cannot be exploited. Their propagation therefore still requires dealing with error terms in the region near the light cone, where only $t-|x|$ decay is available. We shall return to this issue later, after discussing how this difficulty is overcome in physically relevant Vlasov systems.

To circumvent the difficulty raised by \ref{it2}, we control the derivatives of $f$ by exploiting a triangular structure in the commuted equations. Denoting by $\T_{\mathbf{A}}$ the Vlasov operator in \eqref{VW}, and ignoring the difficulties caused by the weak decay near the light cone, we observe that, schematically
    \[ \big| \T_{\mathbf{A}} \big( \nabla_{t,x} f \big) \big| \lesssim \langle v \rangle^{\delta-1}\langle t \rangle^{-3/2} \big| \nabla_{t,x} f \big|(t,x)+O \big(\langle t \rangle^{-2} \big), \qquad \qquad \big| \T_{\mathbf{A}} \big( \widehat{Z} f \big) \big| \lesssim \langle v \rangle^{\delta-1}\langle t \rangle^{-1} \big| \nabla_{t,x} f \big|(t,x)+O \big(\langle t \rangle^{-3/2} \big), \]
    where, for instance, $\widehat{Z}=t \partial_{x^i}+x^i \partial_t+\langle v \rangle \partial_{v^i}$ commutes with the free transport operator $\langle v \rangle \partial_t + v \cdot \nabla_x$.

\subsection{Previous works on the perturbation of vacuum for relativistic Vlasov systems}\label{Subseccitation}

The analysis of small data solutions to the Vlasov-Maxwell system was initiated by Glassey-Strauss \cite{GSt} for compactly supported data. Later, Schaeffer \cite{Sc} relaxed the compact support assumption in $v$. Eventually, all compact support assumptions on the initial data were removed in \cite{dim3,Wang}. Let us emphasize that, unlike the previous works, these approaches allow one to derive sharp decay estimates for high-order derivatives of the solutions.

In \cite{WeiYang}, Wei-Yang obtained a global existence result for large electromagnetic fields. Their method, however, does not provide information on the derivatives of $f$. Finally, in \cite{scat}, we proved an asymptotic stability result without any compact support assumption, allowing for large electromagnetic fields and showing that an arbitrarily large number of derivatives of the solutions decay at the optimal rate.

Recently, the exact large-time behavior of solutions has been determined. The distribution function exhibits modified scattering \cite{scat,BAP,Emile}, whereas the electromagnetic field undergoes linear scattering \cite{scat}. Moreover, the corresponding scattering states satisfy a system of constraint equations \cite{scatmap}. Such an asymptotic behavior was previously identified for small data solutions to the Vlasov-Poisson system \cite{scattPoiss}.

For the Vlasov-Nordström system, \cite{VNsmall} establishes optimal decay estimates for small data solutions under a compact support assumption. Higher order estimates were later derived by Fajman-Joudioux-Smulevici \cite{FJS,FJS2} under a compact support assumption in $x$, and subsequently by Wang \cite{Xuecheng} for polynomially decaying data.

Finally, the asymptotic stability of Minkowski spacetime was first established by \cite{ReinRendall} under spherical symmetry and a compact support assumption on the initial distribution function. A proof without symmetry was later obtained in \cite{LindbladTaylor,FJS}, together with optimal decay estimates for high-order derivatives. In the latter work, the support of the distribution function may be unbounded in the momentum variable. More recently, Wang \cite{WangMinko} removed all compact support assumptions on the initial data.

\subsection{The null condition}\label{Subsecnullcompo}
 Let us briefly recall some properties of small data solutions to semilinear wave equations of the form $\Box \phi = \partial \phi \cdot \partial \phi$ on $\R_t \times \R^d_x$.
 \begin{itemize}
     \item In dimensions $d \geq 4$, linear waves decay sufficiently fast to guarantee the global existence of these solutions.
     \item In dimension $d=3$, \cite{FJ81} showed that the small data solutions may blow-up in finite time. For instance this occurs when $\partial \phi \cdot \partial \phi = |\partial_t \phi|^2$.
     \item In $3d$, \cite{Christo86,Klai86} identified a sufficient condition on the nonlinearity, now referred as the \textit{null condition}, which guarantees that small data solutions exist globally in time and exhibit optimal decay.

     The key idea is that the solutions to $\Box \psi =0$ disperse in the directions $\mathcal{T} \coloneqq \{ L, e_\theta ,e_\varphi \}$ tangential to the light cone, where $L \coloneqq \partial_t + \partial_r$, $e_\theta \coloneqq r^{-1} \partial_\theta$ and $e_\varphi \coloneqq  r^{-1} \sin^{-1} (\theta ) \partial_{\varphi}$. As a consequence, if $\Box \psi =0$, then 
     \[ \forall \, T \in \mathcal{T}, \quad \big| \nabla_T \psi \big| (t,x) \lesssim \langle t+|x| \rangle^{-2}, \qquad \qquad \qquad \big| \nabla_{t,x} \psi \big| (t,x) \lesssim \langle t+|x| \rangle^{-1}.\]
     The null condition states that the nonlinearity satisfy $|\partial \phi \cdot \partial \phi| \lesssim \sum_{T \in \mathcal{T}} |\nabla_T \phi| |\nabla_{t,x} \phi|$.
 \end{itemize}

 In all the works concerning small data solutions to the Vlasov-Maxwell, Vlasov-Nordström and Einstein-Vlasov system in which
 \begin{enumerate}[label = (\alph*)]
     \item \label{item1} the initial distribution is not compactly supported in both the spatial $x$ and the momentum $v$ variables,
     \item \label{item2} estimates are propagated on the derivatives of $f$,
 \end{enumerate}
 a geometric feature of the nonlinearity $\nabla \mathbf{A} \cdot \nabla_v f$, reminiscent of the null condition, plays a crucial role.
 \begin{Rq}\label{Remconj}
 Among the results mentioned in Section \ref{Subseccitation}, only \cite{dim3,Wang,scat,FJS,FJS2,Xuecheng,FJS3,WangMinko} are concerned by both \ref{item1}--\ref{item2}. In contrast, in \cite{WeiYang}, where the requirement \ref{item1} is satisfied but not \ref{item2}, the null structure in the Vlasov equation is not exploited. This leads us to formulate Conjecture \ref{Conj} below.
 \end{Rq}

 The manifestation of the null structure in these three Vlasov systems can be summarise as follows. When expanded in the null frame $(\underline{L},L,e_\theta, e_\varphi)$, where $\underline{L} \coloneqq \partial_t - \partial_r$, the nonlinearity $\nabla \mathbf{A} \cdot \nabla_v f$ can be decomposed into terms containing at least one of the following factors.
 \begin{itemize}
     \item A good component of the electromagnetic field $(E,B)$ or of the metric $g$. For instance, $\frac{x}{|x|} \cdot E(t,x)$ and $\partial_{t,x} g_{LT}$, for any $T \in \mathcal{T}$ decay (almost) like $\langle t+|x| \rangle^{-2}$.
     \item A good component of the momentum vector, namely $v_L=-\langle v \rangle + \frac{x}{|x|} \cdot v$ or $v_{e_A}$, for $A \in \{ \theta , \varphi \}$. For particles escaping to infinity at a speed close to the speed of light $c=1$, one has $|v_L| \approx \frac{1}{2|v|}$ and $|v| \gg 1$. Moreover, $|v_{e_A}|^2 \lesssim \langle v \rangle \, |v_L|$.
     \item The derivative $\frac{x}{|x|} \cdot \nabla_v f$ which, in contrast with $\partial_{v^i} f$, does not exhibit growth in time when $t \sim |x|$.
 \end{itemize}
In fact, the structure is richer than this brief description suggests. We refer to \cite[Lemma~4.1]{massless} and \cite[Proposition~5.14]{EVmassless} for a complete description of the null structure of $\nabla \mathbf{A} \cdot \nabla_v f$ in the case of the Vlasov-Maxwell and Einstein-Vlasov systems.

\begin{Rq}
The conservative structure of the force fields in the Vlasov--Maxwell and Einstein--Vlasov systems is not, by itself, responsible for the validity of the known small-data global existence and stability results. Indeed, inserting a factor $\widehat{v}^1$ in front of the force field would break this conservative structure, while leaving unchanged the essential estimates underlying these results. In particular, the corresponding small-data results are expected to remain valid for the resulting modified systems.
\end{Rq}

\begin{Rq}
A major difficulty in proving global existence for classical solutions to such systems is the control of large velocities. From this perspective, the Vlasov-Maxwell system, which corresponds to $\delta=0$, appears more favorable than the Vlasov-Nordström system. Nevertheless, an additional null structure in the latter allowed \cite{Calogero06} to prove global existence for large data, whereas the corresponding problem remains open for the Vlasov-Maxwell system (see however \cite{LukStrain} for continuation criteria).
\end{Rq}

\subsection{Key ideas}\label{SubsecKeyidea}

In this paper, we do not impose any assumption on the nonlinearity in \eqref{VW}. It is therefore not clear whether the stability results discussed above persist outside the compactly supported regime. It turns out that the cases $\delta =0$ and $\delta=1$ differs regarding some aspects.

\subsubsection{The case $\delta=0$}\label{subsubsec0} The ODE satisfied by the velocity characteristics allows us to lose one power of $v$. More precisely, we will observe that
\[ \big| \dot{\mathcal{V}} \big| \lesssim \langle t \rangle^{-1} \langle t-|\mathcal{X}| \rangle^{-1} \lesssim \langle \mathcal{V} \rangle \, \langle t \rangle^{-2}+  \big| \mathcal{V}_L \big| \langle t-|\mathcal{X} | \rangle^{-2}, \]
since $1 \lesssim \langle v \rangle |v_L|$. Then, the good weight $v_L$ will allow us to exploit the decay in $t-|x|$ by parameterising the characteristics by $u=t-|x|$. We will then show that $\mathcal{V}$ remains bounded by $C \langle \mathcal{V}_{t=0} \rangle$ for a constant $C>0$ independent of $(\mathcal{X},\mathcal{V})_{t=0}$. One can then use this property to control the derivatives of $f$. Returning to problem \ref{it2} discussed in Section \ref{Subsecdiff}, we have, using again $1 \lesssim \langle v \rangle |v_L|$ and neglecting favorable error terms,
\begin{equation}\label{eq:introderiv}
\big| \T_{\mathbf{A}} \big( |\nabla_{t,x} f| \big) \big| \lesssim |v_L| \langle t-|x| \rangle^{-3/2}\big| \nabla_{t,x} f \big|(t,x), \qquad \qquad \big| \T_{\mathbf{A}} \big( |\widehat{Z} f| \big) \big| \lesssim |v_L|\langle t-|x| \rangle^{-1} \big| \nabla_{t,x} f \big|(t,x). 
\end{equation}
Then, by Duhamel's principle and a triangular Grönwall argument, we obtain the rough bound $|\nabla_{t,x} f |(t,x) \leq  |\nabla_{t,x} f|(0,\mathcal{X}_{t=0},\mathcal{V}_{t=0}) e^{C \langle \mathcal{V}_{t=0} \rangle}$, together with a slightly weaker estimate for $\widehat{Z}f$, due to a logarithmic loss. In particular, this suggests that asymptotic stability should hold for $(\text{Vw--0})$ for initial data exponentially decaying in $v$. As the size of the scalar fields $A_I$ increases, a stronger exponential decay is required.
 
It turns out that this exponential loss can be substantially reduced by first studying the commuted system for $(\nabla_{t,x} f, \langle v \rangle \nabla_v f)$. Indeed, one has
\begin{equation}\label{eq:latoutintro}
\big| \mathbf{T}_{\mathbf{A}} \big( |\nabla_{t,x}f| \big) \big|  \lesssim \mathbf{a}(t,x,v)  \langle v \rangle | \nabla_{v}f |  , \qquad \qquad \big| \mathbf{T}_{\mathbf{A}} \big( \langle v \rangle |\nabla_v f| \big) \big|   \leq    \big|\nabla_x f \big|(t,x,v) + \text{l.o.t.}, 
\end{equation}
where $\mathrm{l.o.t.}$ stands for lower order terms and, with $\Lambda$ measuring the size of the scalar fields, 
\[ \mathbf{a}(t,x,v) \coloneqq \frac{\Lambda}{\langle v \rangle (1+t) \, \langle t-|x| \rangle^{\frac{3}{2}}}. \] 
Recall that even at the linearised level, the $v$-derivatives grow linearly in $t$. The idea is then to close the estimate by controlling
\begin{equation}\label{defcalF}
 \mathcal{F} (t,x,v) \coloneqq   \big| \nabla_x f \big|(t,x,v)+ \sqrt{\mathbf{a}(t,x,v)}  \langle v \rangle \big| \nabla_v f \big|(t,x,v). 
\end{equation}
This choice of $\mathcal{F}$ is motivated by an approximate diagonalisation of the transport system \eqref{eq:latoutintro}, whose eigenvalues are of size $\pm \sqrt{\mathbf{a}}$. The transport equation satisfied by $\mathcal{F}$ will then have better properties,
\[  \T_{\mathbf{A}} (\mathcal{F} )  \lesssim \sqrt{\mathbf{a}(t,x,v)} \mathcal{F}+\T_{\mathbf{A}} \big( \sqrt{\mathbf{a}} \big) \mathcal{F}. \] 
A careful analysis of the integral of $\sqrt{\mathbf{a}(t,x,v)}$ along the nonlinear flow will allow us to reduce the exponential growth to $e^{C\sqrt{\Lambda} \langle \mathcal{V}_{t=0} \rangle^{\frac{1}{2}}}$.

\subsubsection{The case $\delta =1$}\label{subsubsec1} The extra power of $v$ deteriorates the estimate of the velocity characteristics. More precisely, we merely have
\[ \big| \dot{\mathcal{V}} \big| \lesssim \langle \mathcal{V} \rangle \, \langle t \rangle^{-1} \langle t-|\mathcal{X}| \rangle^{-1} , \]
so that we merely obtain $|\mathcal{V}(t)| \leq \langle \mathcal{V}_{t=0} \rangle \, t^{C}$. Then, we have two possibilities here.
\begin{itemize}
    \item If the scalar fields are allowed to be large, then the constant $C$ may also become large, and we show that asymptotic stability can fail. See in particular Proposition \ref{ProInstaLambdaLarge}.
    \item If the scalar fields are of size $\Lambda \ll 1$, $t^C$ is a small polynomial growth. Using this estimate, together with $1 \lesssim \langle v \rangle |v_L|$, we will be able to show that $|\mathcal{V}(t)| \leq C' \langle \mathcal{V}_{t=0} \rangle^{1+\kappa}$, for some constant $\Lambda \lesssim  \kappa <1/2$.
\end{itemize}
To control the derivatives, we note that \eqref{eq:introderiv} holds with one additional power of $\langle v \rangle$ in the right hand sides. It yields to 
\[ \big| \nabla_{t,x} f \big| (t,x,v) \leq \big| \nabla_{t,x} f \big| \big(0, \mathcal{X}(0), \mathcal{V}(0) \big) e^{C' \Lambda \langle \mathcal{V}(0) \rangle^{2+2\kappa} } .\]
As in the $\delta=0$ case, the exponential growth can in fact be reduced to $e^{C' \sqrt{\Lambda} \langle \mathcal{V}(0) \rangle^{1+\kappa} }$, which allows to treat distribution functions with Maxwellian decay.

\subsubsection{Deriving optimal decay estimates}\label{Subsubsecmachin}

To close the bootstrap argument, it is crucial to derive decay estimates for the source term in the wave equations that are compatible with the losses incurred by certain derivatives, without introducing any additional loss. We have identified two possible approaches.
\begin{enumerate}[ label = (\Roman*)]
    \item \label{traou1} Since the top order derivatives are more difficult to handle, we could exploit a gain of regularity in the wave equations $\Box A_I= \pmb{\rho}_I [f]$. By adapting the result \cite{GlStrauss} of Glassey-Strauss, we could close the estimates for $\nabla_{t,x}A_I$ and $\nabla_{t,x}^2 A_I$ while, roughly speaking, only requiring optimal decay for $\int_v \langle v \rangle^5 |f| \dr v$. This gain of regularity is reminiscent of the elliptic regularity available in the Vlasov--Poisson system.
    \item \label{traou2} Use a robust method for deriving optimal decay estimates for momentum averages that can, in particular, handle the top order derivatives. 
\end{enumerate}
We believe that \ref{traou2} is simpler (compare with \cite{scat}, where we followed \ref{traou1}). To implement this strategy, we seek a weight function $\pmb{\omega}$ such that
    \begin{equation}\label{dernierlabel}
    \int_{\R^3_v} \big|h(t,x,v) \big| \dr v  \lesssim \int_{\R^3_v} \frac{\dr v}{\pmb{\omega}(t,x,v)} \sup_{w \in \R^3_v}\pmb{\omega}(t,x,w) \big|h(t,x,w) \big|  \lesssim \frac{1}{\langle t+|x| \rangle^3} \sup_{w \in \R^3_v} \pmb{\omega}(t,x,w) \big|h(t,x,w) \big|, 
    \end{equation}
    and which does not grow along the flow associated with $\T_{\mathbf{A}}$. It turns out that the two most natural choices cannot be used here.
    \begin{itemize}
        \item One might first try $\pmb{\omega}(t,x,v) = \langle x -t\widehat{v} \rangle^4 \langle v \rangle^5$, which is conserved along the flow associated with $\partial_t+\widehat{v} \cdot \nabla_x$. The desired decay rate then follows from the change of variables $y(v)=x-t\widehat{v}$. However, since $f$ follows a modified scattering dynamics, $\pmb{\omega}$ grows logarithmically along the nonlinear characteristics.
        \item One could instead use $\pmb{\omega}(t,x,v) = \langle \mathcal{X} (0,t,x,v) \rangle^4 \langle v \rangle^n$, for $n$ sufficiently large. However, consistently with the discussion in Sections \ref{subsubsec0}--\ref{subsubsec1}, we are unable to prove uniform boundedness in $(x,v)$ of $t^{-1}\nabla_v \mathcal{X} (0,t,x,v)$. This prevents us from deriving the $t^{-3}$ decay of the integral in \eqref{dernierlabel} through the change of variables $y(v)=\mathcal{X} (0,t,x,v)$.
    \end{itemize}
    Recall that \eqref{eq:intromodscatt} suggests the existence of a function $\mathscr{C}(v)$ such that $x-t\widehat{v}-\mathscr{C}(v) \log(1+t)-\mathcal{X} (0,t,x,v)$ remains uniformly bounded in $t$. Even if $\nabla_v\mathscr{C}(v)$ turns out to grow in $v$, one may still expect to obtain the decay in \eqref{dernierlabel}, with $\pmb{\omega}(t,x,v) \coloneqq \langle x-t\widehat{v}-\mathscr{C}(v) \log(1+t) \rangle^4 \langle v \rangle^n$, through the change of variables $y=x-t\widehat{v}-\mathscr{C}(v) \log(1+t)$ and by exploiting the $\langle v \rangle^n$ weight. Indeed, in the Jacobian determinant, $\nabla_v\mathscr{C}(v)$ appears with the decaying factor $t^{-1}\log(t)$, and thus gives only a small contribution for, say, $\langle v \rangle^{n-4} \leq t^{3}$. However, as we shall see, $\mathscr{C}(v)$ is intrinsically defined in terms of the asymptotic state of the solution and therefore cannot be used in the proof of global existence. To overcome this difficulty, we construct a dynamical correction $\mathscr{C}_t(v)$ satisfying $\mathscr{C}_t (v) \to \mathscr{C}(v)$ as $t \to +\infty$.

\subsection{Notation} Throughout the paper, we use the notation $A \lesssim B$ to mean that $A \leq CB$ for some some constant $C>0$, possibly depending on $\Lambda$ and the parameters $\eta, \kappa$ introduced below in Section \ref{SecAsympstab}. In some instances, we will make the constant $C$ explicit, writing $C[B_1,\dots,B_k]$ to indicate that it depends on the parameters $(B_1, \dots , B_k)$.

\subsection{Main  result} 

We now describe the asymptotic behavior of small data solutions to \eqref{VW} arising from strongly decaying initial data.

\begin{Th}\label{Th}
   Let $K \in \mathbb{N}^*$, $f_0 \in W^{1,\infty}(\R^3_x \times \R^3_v)$ and $(A_I^0,\partial_t A_I^0) \in W^{2,\infty}(\R^3)\times W^{1,\infty}(\R^3)$ for any $I \in \llbracket 1 , K \rrbracket$. 
   \begin{itemize}
       \item If $\delta =0$, we assume that there exist $0< \varepsilon \leq \Lambda$ and $B>0$ such that
  \begin{equation}\label{eq:assumpA}
\begin{aligned}
  \sup_{1 \leq I \leq K}\; \sup_{|\gamma| \leq 2}  \;  \sup_{x \in \R^3} \; \langle x \rangle^{2+|\gamma|} \big| \partial_{t,x}^\gamma A_I^0 (x) \big| & \leq \Lambda , \\
    \sup_{(x,v) \in \R^3_x \times \R^3_v}  \, \langle x \rangle^{12}  \langle v \rangle^{160} e^{B \sqrt{\Lambda} \, \langle v \rangle^{\frac{1}{2}}} \big| \nabla_{x,v} f_0 (x,v) \big|+\langle x \rangle^8 \langle v \rangle^{150} \big|  f_0 (x,v) \big| & \leq \varepsilon .  
\end{aligned} \tag{$\mathrm{Hyp}_{\delta=0}$}
\end{equation}
   There exist two constants $B_0[\Lambda]>0$ and $\varepsilon_0[\Lambda]>0$ such that, if $B \geq B_0$ and $\varepsilon \leq \varepsilon_0$, then the unique solution $(f,A_1, \dots , A_K)$ to $(\text{Vw}-0)$ arising from the data $(f_0,A_{1}^0, \partial_t A_{1}^0, \dots ,A_{K}^0, \partial_t A_{K}^0 )$ is global in time.
   \item  If $\delta =1$, we assume that there exist $0< \varepsilon \leq \Lambda $, $B>0$ and $\kappa >0$ such that
 \begin{equation}\label{eq:assumpB}
 \begin{aligned}
  \sup_{1 \leq I \leq K}\; \sup_{|\gamma| \leq 2}  \;  \sup_{x \in \R^3} \; \langle x \rangle^{2+|\gamma|} \big| \partial_{t,x}^\gamma A_I^0 (x) \big| & \leq \Lambda ,   \\
    \sup_{(x,v) \in \R^3_x \times \R^3_v}  \, \langle x \rangle^{12}  \langle v \rangle^{160} e^{B \sqrt{\Lambda} \, \langle v \rangle^{\frac{1}{2}}} \big| \nabla_{x,v} f_0 (x,v) \big|+\langle x \rangle^8 \langle v \rangle^{150} \big|  f_0 (x,v) \big| & \leq \varepsilon .  
\end{aligned} \tag{$\mathrm{Hyp}_{\delta=1}$}
\end{equation} 
There exist constants $\Lambda_0[\kappa] , \, B_0[\kappa], \, \varepsilon_0[\kappa] >0$ such that, if $\Lambda \leq \Lambda_0$, $B \geq B_0$ and $\varepsilon \leq \varepsilon_0$, then the unique solution $(f,A_1, \dots , A_K)$ to $(\text{Vw}-1)$ arising from the data $(f_0,A_{1}^0, \partial_t A_{1}^0, \dots ,A_{K}^0, \partial_t A_{K}^0 )$ is global in time.
   \end{itemize}
 Moreover, the following asymptotics hold in both cases.
 \begin{itemize}
     \item The scalar fields satisfy linear scattering and pointwise decay estimates. For any $1 \leq I \leq K$, we have
      \begin{align*}
     \forall \, (t,x) \in \R_+ \times \R^3, \qquad \qquad   \big| A_I \big| (t,x) +\langle t-|x| \rangle \big| \nabla_{t,x} A_I \big| (t,x)+ \langle t-|x| \rangle^{2-\frac{1}{8}} \big| \nabla_{t,x}^2 A_I \big| (t,x) \lesssim \Lambda \langle t +|x| \rangle^{-1}.
   \end{align*}
   Furthermore, there exists $A_I^\infty \in L^\infty \big(\R , L^2(\mathbb{S}^2) \big)$ such that
   \[  rA_I \big(r+u,r\omega \big) \rightharpoonup^\ast A_I^\infty \qquad \text{in }  L^\infty_{\mathrm{loc}} \big(\R_u , L^2(\mathbb{S}^2_\omega) \big) \quad \text{as $r \to + \infty$}. \]
   \item The distribution function exhibits modified scattering. There exist $\mathscr{C} \in C^0(\R^3_v,\R^3)$ and $f_\infty \in C^0(\R^3_x \times \R^3_v)$ such that,
   \[ \forall \, (t,x,v) \in [1,+\infty [ \times \R^3_x \times \R^3_v , \qquad \qquad \big| f \big( t,x+t\widehat{v}+\mathscr{C}(v) \log(t),v \big) - f_\infty (x,v)   \big| \lesssim \varepsilon \, \langle t \rangle^{-\frac{1}{3}}, \]
   where $\mathscr{C}$ is a functional of $\int_x f_\infty (x,\cdot ) \dr x$.
 \end{itemize}
\end{Th}

Along the lines of Remark \ref{Remconj}, we expect that asymptotic stability may in fact hold under substantially weaker assumptions on the initial data, provided one establishes stability in a weaker norm that controls $f$ but not its derivatives.
\begin{Conj}\label{Conj}
    Let $f_0 \in C^1$ and $A_1^0,\dots,A_K^0 \in C^2$ be initial data for \eqref{VW}. Assume that 
    \[  \sup_{1 \leq I \leq K} \; \sup_{|\gamma| \leq 2} \; \sup_{x \in \R^3} \; \langle x \rangle^{2+|\gamma|} \big| \partial_{t,x}^\gamma A^0_I (x) \big| \leq \Lambda, \qquad  \sup_{(x,v) \in \R^3_x \times \R^3_v} \langle |x|+|v| \rangle^B \big|f_0(x,v) \big| \leq \varepsilon , \]
    where $\Lambda = \sqrt{\epsilon}$ if $\delta =1$. Then, provided $B \geq 100$ and $\varepsilon>0$ is sufficiently small, the corresponding classical solution is global in time. Moreover, $\nabla_{t,x}A_I$ and $\int_{v} f \dr v$ decay at their optimal rates.
\end{Conj}
In \cite{WeiYang}, the condition $B>9$ is sufficient for the Vlasov-Maxwell system. We expect a slightly stronger assumption for \eqref{VW} when $\delta=0$, and an even stronger one when $\delta =1$, due to Proposition \ref{ProVcharacbis}.

\begin{Rq}
We do not know whether the assumptions on the decay in the momentum variable in Theorem \ref{Th} are optimal. They arise naturally from our control of the characteristics and of the commuted transport equations, and weakening them does not seem accessible by a straightforward refinement of our estimates. Obtaining a result under weaker $v$-decay assumptions, if possible, would likely require a different approach or a substantially more precise analysis of the interaction between the nonlinear characteristics and the scalar field.

In contrast, we expect that Conjecture \ref{Conj} can be proved by minor adaptations of the arguments in \cite{WeiYang}.
\end{Rq}

\subsection{Instability under weaker assumptions} We state here nonlinear instability results driven by a linearly unstable transport mechanism. In particular, the instability mechanism is already present at the linear level. For simplicity, we will focus on a specific case of the systems \eqref{VW}, with $K=1$ scalar field,
\begin{equation}\label{VWspeci}
    \partial_t f + \widehat{v} \cdot \nabla_x f +\langle v \rangle^\delta \partial_{x^1} A \, \partial_{v^1} f =0, \qquad \qquad \Box A = \int_{\R^3_v} f \dr v. \tag{$\mathrm{Vw}^{\mathrm{spec}}$--$\delta$}
\end{equation}
In fact, the exact expression of the source term in the wave equation will play no role here. We introduce first the profile of the initial scalar field that we will use.

\begin{Def}\label{Definsta}
   Let $ \mathcal{A} \in C^3_c (\R,\R_+) $ such that
   \[  \forall \, u \in [-10,-1], \quad \mathcal{A}(u) = e^{-2u}, \qquad \qquad \qquad  \forall \, u \geq 0, \quad \mathcal{A}(u)=0. \]
We denote by $\mathcal{A}_{\mathrm{hom}}$ the spherically symmetric solution to the homogeneous wave equation satisfying 
\[ \mathcal{A}_{\mathrm{hom}}(0,x ) =  \partial_t \mathcal{A}_{\mathrm{hom}} (0, x ) = \frac{\mathcal{A}(-|x|)}{|x|}. \]
 \end{Def}

We now prove a result which shows that the smallness assumption on the scalar field is necessary in the case $\delta=1$ in Theorem \ref{Th} and Conjecture \ref{Conj}. In other words, we show that for $\Lambda$ large enough, the solution $(0,\Lambda \mathcal{A}_{\mathrm{hom}})$ to \eqref{VWspeci} fails to be stable in the sense of Theorem \ref{Th}.

\begin{Pro}\label{ProInstaLambdaLarge}
Let $\delta =1$, and let $(f^\varepsilon , A^\Lambda)$ be the unique solution to \eqref{VWspeci} with initial data $A^\Lambda(0,x)= \partial_t A^\Lambda (0,x)=\Lambda \mathcal{A}(-|x|)/|x|$ and $f^\varepsilon (0,\cdot , \cdot)= \varepsilon f_0$, where $f_0 \in C^\infty (\R^3_x \times \R^3_v , \R_+)$ satisfies $\|f_0\|_{L^1_{x,v}}=1$, and 
\[ f_0(x,v) \neq 0 \qquad \Longrightarrow \qquad |x-(7,0,0)| \leq 1, \quad \;  0 \leq v^2, \, v^3 \leq 1 \leq v^1. \]
There exists $\Lambda_0 >0$ such that, if $\Lambda \geq \Lambda_0$, the following holds. For all $C_{\mathrm{large}}, \, \varepsilon_1>0$, there exists $\varepsilon \leq \varepsilon_1$ such that
 \begin{itemize}
     \item either $(f^\varepsilon , A^{\Lambda})$ blows up in finite time,
     \item or $\|f^\varepsilon(t,\cdot,\cdot) \|_{L^1_{x,v}} \geq \|f^\varepsilon(0,\cdot,\cdot) \|_{L^1_{x,v}} (1+t)^{c \Lambda}$, where $c>0$ is an absolute constant, 
     \item or there exists $(T,x) \in \R_+ \times \R^3$ such that $\langle T \rangle \big| \partial_{x^1} A^\Lambda- \Lambda \partial_{x^1}\mathcal{A}_{\mathrm{hom}} \big|(T,x) \geq C_{\mathrm{large}} \varepsilon$.
 \end{itemize}
\end{Pro}
\begin{Rq}\label{Rqici}
    The third possibility contradicts the estimate \eqref{eq:BA1} established in the proof of Theorem \ref{Th}.
\end{Rq}
\begin{Rq}
    Proposition \ref{ProInstaLambdaLarge} indicates that any extension of the stability result of Luk-Oh \cite{JonathanSungJin} for large dispersive vacuum Einstein spacetimes to the Einstein--Vlasov system must take into account the null structure of the transport operator, unless the distribution function is assumed to have compact support of size sufficiently small compared to the size of the reference solution.
\end{Rq}

The next result shows that, when $\delta=1$, the decay rates allowed by the original function space are not, in general, asymptotically stable. In particular, some loss of decay is unavoidable. To quantify this loss, we introduce the following class of weights,
\[ \mathcal{W}_b \coloneqq \big\{ W \in C^1 \big(\R_+,[1,+\infty[ \big) \; \big| \; W(s) \to +\infty \;\text{as $s \to +\infty$}, \; \; sW'(s)/W(s) \geq b \big\}, \qquad \qquad b>0,  \]
which includes, in particular, polynomial and exponential weights.
\begin{Pro}\label{ProInstaInidecay}
    Let $\delta =1$, and let $(f , A)$ be the unique solution to \eqref{VWspeci} with initial data 
    \[ A(0,x)= \partial_t A (0,x)=\Lambda \mathcal{A}(-|x|)/|x|, \qquad \qquad f (0,\cdot , \cdot)= \varepsilon f_0, \quad  f_0 \in C^1 (\R^3_x \times \R^3_v , \R).\]
    Assume that $\Lambda$ and $\varepsilon$ are sufficiently small so that $(f,A)$ falls within the framework of Theorem \ref{Th}, and that $\varepsilon / \Lambda$ is sufficiently small. Then, for every $b >0$ and $W \in \mathcal{W}_b$, there exists an absolute constant $c>0$ such that
    \[\forall \, \tau \geq 1, \qquad \qquad   W \big( \langle \mathcal{V}_\tau \rangle \big) \big| f (\tau,\mathcal{X}_\tau,\mathcal{V}_\tau ) \big| \geq \langle \tau \rangle^{2bc \Lambda } W ( \langle \tau \rangle) \big|f_0(7,0,0,\tau,0,0) \big| , \]
    where $(\mathcal{X}_\tau , \mathcal{V}_\tau)$ denotes the value at time $\tau$ of the characteristic that starts from $(7,0,0,\tau,0,0)$ at time $0$.
\end{Pro}
\begin{Rq}
  The proof of Proposition \ref{ProInstaInidecay} actually yields a more general statement.
\end{Rq}
\begin{Rq}
A consequence of Proposition \ref{ProInstaInidecay} is that we do not expect it to be possible to work in a functional framework in which the initial data and the scattering state belong to the same weighted Banach space, as in \cite{VPwaveop} for the Vlasov-Poisson system. In the particular case of the Vlasov--Maxwell and Einstein-Vlasov systems, however, we believe that the null structure of the equations should make such a result possible (see \cite[Remark~1.10]{VPwaveop}).
\end{Rq}

\section{Preparatories}

\subsection{Weight functions}

As explained in Section \ref{SubsecKeyidea}, the ingoing null component of the velocity $4$-vector will play a crucial role for us. For convenience, we will work with a rescaled quantity,
\begin{equation}\label{defvLbar} \widehat{v}^{\underline{L}} \coloneqq \frac{1}{2} \bigg( 1- \frac{x}{|x|} \cdot \frac{v}{\langle v \rangle} \bigg).
\end{equation}
\begin{Rq}In the null frame $(\underline{L},L,e_\theta, e_\varphi)$ introduced in Section \ref{Subsecnullcompo}, we can decompose $\mathbf{v}=(\langle v \rangle , v^1,v^2,v^3)$ as $\mathbf{v}=v^{\underline{L}} \underline{L}+v^L L+v^{e_\theta} e_\theta+v^{e_{\varphi}} e_\varphi$. Then, one can check that $\widehat{v}^{\underline{L}}= v^{\underline{L}} / \langle v \rangle$ and, by lowering indices with respect to the Minkowski metric, $\widehat{v}^{\underline{L}}=-\frac{1}{2}v_L$.
\end{Rq}

We recall classical inequalities involving the weights $\langle v \rangle$ and $\langle x-t \widehat{v} \rangle$, which are conserved along the flow of the free transport operator $\T_0$. The second one is useful for exploiting the strong decay of massive Vlasov fields near and outside the light cone $t=|x|$. For a proof, we refer for instance to \cite[Lemma~2.9]{FJS}.

\begin{Lem}\label{Lemweight}
For all $(t,x,v) \in \R \times \R^3_x \times \R^3_v$, we have
\[ 1 \leq 4\langle v\rangle^2 \widehat{v}^{\underline{L}}, \qquad \qquad \langle t+|x|\rangle \lesssim \langle t-|x|\rangle \langle v\rangle^2 + \langle x-t\widehat{v}\rangle \langle v \rangle^2. \]
\end{Lem}

Finally, some of our decay estimates will rely on the change of variables $v \mapsto x-t \widehat{v}$. We then recall \cite[Lemma~2.9]{scat}.

\begin{Lem}\label{Lemcdv}
   Let $\widecheck{\, \; \, }$ be the operator defined on the domain $\{ y \in \R^3 \; | \; |y| <1 \}$ by
   \[ \widecheck{y} \coloneqq \frac{y}{\sqrt{1-|y|^2}}.\]
   We have $\widecheck{\widehat{v}}=v$ and $\widehat{\widecheck{y}}=y$ for all $v \in \R^3_v$ and $|y|<1$. Moreover, the Jacobian determinant of $v \mapsto t\widehat{v}$ equals $t^{-3} \langle v \rangle^5$.
\end{Lem}

\subsection{Commutators}

We will derive estimates on both the fields and the distribution function using vector field methods. These kinds of approaches go back to Klainerman \cite{Kl85}, and were adapted to Vlasov equations by Fajman-Joudioux-Smulevici \cite{FJS}.

We will commute the wave equations using the following weighted derivatives.

\begin{Def}
Let $\mathbb{K}$ be the set composed of the vector fields
\[ \partial_t, \qquad \partial_{x^i}, \qquad \Omega_{0i} \coloneqq t\partial_{x^i}+x^i\partial_t, \qquad \Omega_{jk} \coloneqq x^j\partial_{x^k}-x^k\partial_{x^j}, \qquad S \coloneqq t\partial_t+ \sum_{1 \leq \ell \leq 3} x^\ell\partial_{x^\ell}=t\partial_t+r\partial_r, \]
where $1\leq i \leq 3$ and $1 \leq j<k\leq 3$.
\end{Def}

The elements $Z \in \mathbb{K} \setminus \{S\}$ are Killing vector fields, and then generate isometries of the Minkowski space, whereas $S$ is merely conformal Killing. In particular, we have
\begin{equation}\label{eq:ComBox}
\forall \, Z \in \mathbb{K}\setminus\{ S \}, \qquad [\Box,Z]=0,  \qquad \qquad \qquad [\Box,S]=2\Box.
\end{equation}

We now introduce the weighted derivatives that we will use to commute the Vlasov equation.

\begin{Def}
Let $\Pp$ be the set composed of
\[ \partial_t, \qquad \partial_{x^i}, \qquad \widehat{\Omega}_{0i} \coloneqq t\partial_{x^i}+x^i\partial_t+v^0\partial_{v^i}, \qquad \widehat{\Omega}_{jk} \coloneqq x^j \partial_{x^k}-x^k \partial_{x^j}+v^j \partial_{v^k}-v^k \partial_{v^j},
\]
where $1 \leq i \leq 3$ and $1 \leq j < k \leq 3$. Let also $\Pp_S \coloneqq \Pp \cup \{S\}$.
\end{Def}

The elements $\widehat{Z} \in \Pp$ are the complete lifts of the corresponding vector fields $ Z\in \mathbb{K} \setminus \{ S \}$ (see \cite[Section~2G]{FJS} for more details). They have good commutation properties with the linear transport operator $\T_0=\partial_t+\widehat{v} \cdot \nabla_x$. Indeed,
\begin{equation}\label{eq:ComT}
\forall \, \widehat{Z} \in \Pp, \qquad \big[ \langle v\rangle \T_0,\widehat{Z} \big]=0,
\qquad \big[ \langle v\rangle \T_0,S \big]= \langle v\rangle \T_0.
\end{equation}

We now state the commutation formula. For this, let us denote the Vlasov operator as
\[ \T_{\mathbf{A}} \coloneqq \partial_t+\widehat{v} \cdot \nabla_x +\nabla \mathbf{A} \cdot \nabla_v. \]

\begin{Pro}\label{ProCom}
Let $(f,A_1,\ldots,A_K)$ be a sufficiently regular solution to \eqref{VW}. Then,
\begin{itemize}
\item for any translation $\partial_{x^\nu}$, we have
\[
\mathbf{T}_{\mathbf{A}} \big( \partial_{x^\nu}f \big) = -\partial_{x^\nu} \big( \nabla \mathbf{A} \big) \cdot \nabla_v f, \qquad \qquad \Box \partial_{x^\nu}A_I = \pmb{\rho}_I \big[\partial_{x^\nu}f \big]. \]
\item Consider now $Z=\Omega_{\mu k}$ and $\widehat Z=\widehat{\Omega}_{\mu k}$, where $0 \leq \mu < k \leq 3$, or $Z=\widehat{Z}=S$. Then, we have
\[ \mathbf{T}_{\mathbf{A}} \big(\widehat{Z} f \big)=\nabla \mathbf{A}_Z \cdot \nabla_v f,
\qquad \Box Z A_I = \pmb{\rho}_I \big[\widehat{Z} f \big]+\int_{\R^3_v}\langle v\rangle^\delta P_I^Z(\widehat{v})f \dr v, \]
where $P_I^Z \in \R_3[\widehat{v}]$ is a polynomial, which may be equal to $0$, and $\nabla \mathbf{A}_Z$ is given by
\[ \big[\nabla\mathbf{A}_Z \big]^i(t,x,v) = \sum_{1 \leq I \leq K} \sum_{0 \leq \mu \leq 3} \langle v\rangle^\delta Q_I^{i,\mu}(\widehat{v}) \partial_{x^\mu} Z A_I (t,x) +
\langle v\rangle^\delta R_I^{i,\mu}(\widehat{v})\partial_{x^\mu} A_I (t,x), \qquad i\in\llbracket 1,3\rrbracket, \]
where $R_I^{i,\mu}\in\R_3[\widehat{v}]$ are polynomials, while the polynomials $Q_I^{i,\mu}$ are those introduced in \eqref{defA}.
\end{itemize}
\end{Pro}

\begin{proof}
To commute the Vlasov equation, we use \eqref{eq:ComT}, that $[\partial_{x^\mu},\partial_{x^\nu}]=0$, and $[\partial_{x^\mu},Z]=0$ or $[\partial_{x^\mu},Z]=\pm\partial_{x^\lambda}$ for any $Z \in \mathbb{K}$. We also use $[\partial_{v^i},\partial_{x^\nu}]=0$ and that, for any $\widehat{Z} \in \Pp_S$, $[\partial_{v^i},\widehat{Z}]$ is equal either to $0$, $\partial_{v^k}$ or $\widehat{v}^i\partial_{v^k}$.

For the wave equations, we use \eqref{eq:ComBox}, and we perform integration by parts to deal with $Z \pmb{\rho}_I [f] = \pmb{\rho}_I[Z f]$.
\end{proof}

Our method to prove boundedness for the solutions will also require to commute by $\partial_{v^i}$.

\begin{Lem}\label{LemComdv}
    Let $(f,A_1,\ldots,A_K)$ be a sufficiently regular solution to \eqref{VW}. Then, for any $i \in \{1,2,3\}$,
    \[ \T_{\mathbf{A}} \big( \langle v \rangle \, \partial_{v^i} f \big) = -  \partial_{x^i} f + \widehat{v}^i \widehat{v} \cdot \nabla_x f - \langle v \rangle \,  \partial_{v^i} \big( \nabla \mathbf{A} \big) \cdot   \nabla_v f+\nabla \mathbf{A} \cdot  \widehat{v} \,\partial_{v^i}f . \]
\end{Lem}
\begin{proof}
The first two terms are generated by the commutation of the free transport operator $\partial_t + \widehat{v} \cdot \nabla_x $ with $\partial_{v^i}$, the third one arises from the nonlinearity $\nabla \mathbf{A} \cdot \nabla_v$ and the last one from $\T_{\mathbb{A}}(\langle v \rangle )=\nabla \mathbf{A} \cdot  \widehat{v}$.   
\end{proof}
Finally, it will sometimes be useful to express the derivatives of $f$ composed with the linear flow to $\widehat{Z}f$.

\begin{Lem}\label{Lemrelftoh}
Let $f : \R_+ \times \R^3_x \times \R^3_v \to \R$ be a sufficiently regular function and $h(t,x,v) \coloneqq f(t,x+t\widehat{v},v)$. The following relations hold, for all $(t,x,v) \in \R_+ \times \R^3_x \times \R^3_v$,
$$ \big[ \partial_{x^k} f \big](t,x+t\widehat{v},v) = \partial_{x^k} h(t,x,v), \quad \; \; \big[ S f \big] (t,x+t\widehat{v},v)  = S h(t,x,v), \quad \; \; \big[ \widehat{\Omega}_{ij} f \big](t,x+t\widehat{v},v)  = \widehat{\Omega}_{ij} h(t,x,v),$$
for any $1 \leq i < j \leq 3$ and $ 1 \leq k \leq 3$. For the time derivative and the Lorentz boosts, we have
\begin{align*}
\big[\partial_t f \big](t,x+t\widehat{v},v) & =\partial_t h(t,x,v) -\widehat{v}\cdot \nabla_x h(t,x,v) ,\\
\big[\widehat{\Omega}_{0k}  f\big](t,x+t\widehat{v},v) & = \langle v \rangle \, \partial_{v^k} h(t,x,v)+x^k\partial_t h(t,x,v)-\big(x^k\widehat{v}+\widehat{v}^k x\big)\cdot \nabla_x h(t,x,v)+\widehat{v}^k S  h(t,x,v) .
\end{align*}
\end{Lem}
\begin{proof}
Since $f(t,x,v)=h(t,x-t\widehat{v},v)$, we have, for any $ 1 \leq k \leq 3$,
\begin{align*}
 \partial_{x^k}f(t,x,v)&= \partial_{x^k}h(t,x-t\widehat{v},v), \qquad \qquad \partial_t f(t,x,v) = \big(\partial_t h -\widehat{v} \cdot \nabla_x h\big)(t,x-t\widehat{v},v) , \\
  \langle v \rangle \, \partial_{v^k}f(t,x,v)&= v^0 \partial_{v^k}h(t,x-t\widehat{v},v)-t\partial_{x^k}h(t,x-t\widehat{v},v)+t\widehat{v}_k \big[ \widehat{v} \cdot \nabla_x h \big](t,x-t\widehat{v},v)  .
\end{align*}
It implies the stated relations.
\end{proof}

\subsection{Glassey-Strauss representation of the scalar fields}

In this section, we establish a decomposition for the derivatives of $A_I$ allowing to control $\partial_{x^\mu}A_I$ by $f$ instead of $\partial_{x^\mu}f$. In the case of the Vlasov-Maxwell system, this result corresponds to \cite[Theorem~3]{GlStrauss}.

\begin{Pro}\label{ProGS}
Let $A \colon [0,T[ \times \R^3 \to \R$ be a solution to the equation 
\[\Box A=\int_{\R^3_v}\langle v\rangle^\delta P(\widehat{v})g \dr v, \] where $P \in \R_3[\widehat{v}]$ is a polynomial and $g \colon [0,T[ \times \R^3_x \times \R^3_v \to \R$ is a sufficiently regular distribution function. For any $0 \leq \mu \leq 3$, we have
\[4\pi\partial_{x^\mu}A = \big[\partial_{x^\mu}A \big]_{\mathrm{data}} +\big[\partial_{x^\mu}A \big]_T + \big[\partial_{x^\mu}A \big]_{\T_0}, \]
where the three terms on the right hand side are defined below, with
$ \omega \coloneqq \frac{y-x}{|y-x|} \in \mathbb{S}^2$.
\begin{itemize}
\item The first term can be explicitly computed in terms of the initial data,
\[ \big[ \partial_{x^\mu}A \big]_{\mathrm{data}}(t,x)=4\pi \big[\partial_{x^\mu}A \big]_{\mathrm{hom}}(t,x) + \frac{1}{t}\int_{|y-x| = t} \int_{\R^3_v} \langle v\rangle^\delta P(\widehat{v}) \frac{ \omega^\mu}{1+\widehat{v} \cdot \omega }g(t-|y-x|,y,v)  \dr v \dr y,   \]
where $\omega^0 \coloneqq 1$ and $[\partial_{x^\mu}A]_{\mathrm{hom}}$ is the unique function satisfying
\[ \Box [ \partial_{x^\mu}A ]_{\mathrm{hom}}=0, \qquad \qquad [\partial_{x^\mu}A]_{\mathrm{hom}}(0,\cdot)=\partial_{x^\mu}A(0,\cdot), \qquad \partial_t [\partial_{x^\mu}A ]_{\mathrm{hom}}(0,\cdot)=\partial_t\partial_{x^\mu}A(0,\cdot).\]

\item The second term is given by
\[ \big[\partial_{x^\mu}A \big]_T(t,x) \coloneqq \int_{|y-x|\le t}\int_{\R^3_v} \mathbf{a}^\mu \big( \widehat{v},\omega \big) g \big( t-|y-x|,y,v \big) \dr v \frac{\dr y}{|y-x|^2}, \]
where the integral kernel is given by
\[ \mathbf{a}^0 \big( \widehat{v},\omega \big) \coloneqq -\frac{|\widehat{v}|^2+\widehat{v}\cdot\omega}{(1+\widehat{v}\cdot\omega)^2}
\langle v\rangle^\delta P(\widehat{v}), \qquad \qquad \qquad \mathbf{a}^i \big( \widehat{v},\omega \big) \coloneqq \omega^i \mathbf{a}^0 \big( \widehat{v},\omega \big)+\frac{\omega^i+\widehat{v}^i}{1+\widehat{v} \cdot \omega}\langle v\rangle^\delta P(\widehat{v}), \quad 1\leq i\leq 3. \]

\item The last term is defined by
\[ \big[\partial_{x^\mu}A \big]_{\T_0}(t,x) \coloneqq \int_{|y-x|\leq t}\int_{\R^3_v}
\mathbf{b}^\mu \big( \widehat{v},\omega \big) \big[ \T_0 g \big] \big(t-|y-x|,y,v \big) \dr v\frac{ \dr y}{|y-x|}, \]
where the integral kernel is given by
\[ \mathbf{b}^0 \big( \widehat{v} , \omega \big) \coloneqq \frac{\langle v\rangle^\delta P(\widehat{v})}{1+\widehat{v}\cdot\omega}, \qquad \qquad \qquad \mathbf{b}^i \big( \widehat{v} , \omega \big) \coloneqq \omega^i\frac{\langle v\rangle^\delta P(\widehat{v})}{1+\widehat{v}\cdot\omega}, \quad 1 \leq i\leq 3. \]
\end{itemize}
\end{Pro}

\begin{proof}
First, we use the inhomogeneous Kirchhoff formula to get, for any $0 \leq \mu \leq 3$,
\[ \partial_{x^\mu}A(t,x) = \big[\partial_{x^\mu}A \big]_{\mathrm{hom}}(t,x) +\frac{1}{4\pi} \int_{|y-x|\leq t}\int_{\R^3_v}
\langle v\rangle^\delta P(\widehat{v})\partial_{x^\mu}g \big(t-|y-x|,y,v \big) \dr v\frac{\dr y}{|y-x|}. \]
Then, the key idea is that we can express $\partial_{x^\mu}$ in terms of derivatives tangential to the light cone and the free transport operator $\T_0=\partial_t+\widehat{v} \cdot \nabla_x$, which is transverse to the light cone. More precisely, let
\[ T_i \coloneqq \partial_{y^i}-\omega^i\partial_t, \qquad \omega^i=\frac{y^i-x^i}{|y-x|}, \]
so that $\partial_{y^i}[g(t-|y-x|,y,v)]=[T_i g](t-|y-x|,y,v)$. Since
\[
\partial_t=\frac{\T_0-\widehat{v}\cdot T}{1+\widehat{v} \cdot \omega }, \qquad \qquad \partial_{x^i}=T_i + \frac{\omega^i}{1+\widehat{v} \cdot \omega}(\T_0-\widehat{v} \cdot T), \]
we obtain
\[ \partial_{x^\mu}A(t,x)- \big[ \partial_{x^\mu}A \big]_{\mathrm{hom}}(t,x)- \big[\partial_{x^\mu}A \big]_{\T_0}(t,x)
= \frac{1}{4 \pi}\int_{|y-x|\leq t} \int_{\R^3_v} \langle v\rangle^\delta P(\widehat{v}) \frac{\widetilde{\mathbf{a}}^\mu(v,\omega)}{|y-x|} \cdot \nabla_y\big[g(t-|y-x|,y,v) \big] \dr v \dr y, \]
where the kernel $\widetilde{\mathbf{a}}^\mu(v,\omega)\in\R^3$ is given by

\[ \widetilde{\mathbf{a}}^0 \coloneqq -\frac{\widehat{v}}{1+\widehat{v}\cdot\omega},
\qquad \qquad \big[\widetilde{\mathbf{a}}^i \big]^j \coloneqq \delta_{i,j}+ \omega^i \big[ \widetilde{\mathbf{a}}^0 \big]^j. \]
We conclude the proof by performing integration by parts, and by observing that, with $\omega^0=1$,
\[ \nabla_y\cdot\left(\frac{\widetilde{\mathbf{a}}^\mu(v,\omega)}{|y-x|}\right) = -\frac{\omega}{|y-x|^2} \cdot \widetilde{\mathbf{a}}^\mu(v,\omega) + \omega^\mu\frac{|\widehat{v}|^2-|\widehat{v}\cdot\omega|^2}{|y-x|^2(1+\widehat{v}\cdot\omega)^2}- \big(1-\delta_{\mu,0} \big)\frac{\widehat{v}^\mu-\omega^\mu \widehat{v} \cdot \omega}{|y-x|^2(1+\widehat{v} \cdot \omega)} , \]
where the last term on the right hand side arises from $\nabla_y \omega^\mu$ and vanish if $\mu =0$.
\end{proof}

To exploit this decomposition, we will need to estimate the integral kernels.

\begin{Lem}\label{Lemkernel}
There exists $C>0$, depending only on $P$, such that, for any $0 \leq \mu \leq 3$ and all $(v,\omega) \in \R^3_v \times \mathbb{S}^2$,
\[
\langle v \rangle \big| \mathbf{a}^\mu \big( \widehat{v},\omega \big) \big|+\big| \nabla_v \mathbf{a}^\mu \big( \widehat{v},\omega \big) \big| \leq C \langle v\rangle^{5+\delta}, \qquad \qquad \langle v \rangle \big| \mathbf{b}^\mu \big( \widehat{v},\omega \big) \big|+\big| \nabla_v \mathbf{b}^\mu \big( \widehat{v},\omega \big) \big| \leq C \langle v\rangle^{3+\delta}.
\]
\end{Lem}
\begin{Rq}
Thanks to the null structure of the Vlasov-Maxwell system, the corresponding integral kernels grow at most like $\langle v\rangle$ (see for instance \cite[Corollary~5.5]{scat}).
\end{Rq}
\begin{proof}
We use that $|1+\widehat{v}\cdot\omega| \geq 1-|\widehat{v}| = \frac{\langle v\rangle^2-|v|^2}{\langle v\rangle(\langle v\rangle+|v|)}$, and $|\nabla_v\widehat{v}|\lesssim\langle v\rangle^{-1}$.
\end{proof}

In the context of \eqref{VW}, $[\nabla_{t,x} A_I]_{\mathrm{data}}$ evolves freely and is decoupled from the nonlinear dynamic. We will then be able to estimate it directly by applying the next result.

\begin{Pro}\label{Prodata}
Assume that there exists $\mathbf{\Lambda} >0$ such that
\[ \sup_{1 \leq |\gamma| \leq 2} \; \sup_{x \in \R^3} \langle x \rangle^{2+|\gamma|} \big| \partial_{t,x}^\gamma A(0,x) \big|+\sup_{(x,v) \in \R^3_x \times \R^3_v } \langle x \rangle^4 \langle v \rangle^7 \big| g (t,x,v) \big| \leq \mathbf{\Lambda} .\]
Then, there exists $\mathbf{C} >0$, depending only on $P(\widehat{v}$, such that, for any $0 \leq \mu \leq 3$ and for all $(t,x) \in \R_+ \times \R^3 $, we have
\[ \big| \big[ \partial_{x^\mu} A \big]_{\mathrm{data}}   \big|(t,x) \leq \mathbf{C} \mathbf{\Lambda} \, \langle t+|x| \rangle^{-1} \langle t-|x| \rangle^{-2}. \]
\end{Pro}
\begin{proof}
 The proof relies on three ingredients. First, recall that if $\Box \phi =0$, we have by Kirchhoff's formula that
$$ \phi(t,x) = \frac{1}{4\pi t^2} \int_{|y-x|=t} \phi(0,y)\dr y+\frac{1}{4\pi t} \int_{|y-x|=t}\frac{y-x}{|y-x|}\cdot \nabla_y \phi(0,y)+ \partial_t \phi(0,y) \dr y.$$
Then one can prove (see for instance \cite[Lemma~$4.1$]{WeiYang}), that if $h \in C(\R^3)$ satisfies $|h|(x) \leq K_0(1+|x|)^{-p}$ for some $p \geq 3$, then
\begin{equation}\label{decayhYang}
\frac{1}{4 \pi} \int_{|y-x|=t}|h|(y) \dr y \leq  \frac{4 K_0 \min (t,t^2 \langle t-|x| \rangle^{-1})}{\langle t+|x| \rangle \, \langle t-|x| \rangle^{-p+2}}.
\end{equation}
Finally, recall from Proposition \ref{ProGS} the expression of $ \big[ \partial_{x^\mu} A]_{\mathrm{data}}$, and use $1+ \widehat{v} \cdot \omega \gtrsim \langle v \rangle^{-2}$.
\end{proof}

\section{Asymptotic stability}\label{SecAsympstab}

Let $\delta \in \{0,1\}$ and let $(f_0,A_{I}^0)$ be initial data for \eqref{VW} satisfying the assumptions \eqref{eq:assumpA} if $\delta=0$, or the assumptions \eqref{eq:assumpB} if $\delta =1$. Let $0 < \eta < 1/8$ and $ T \in \R_+ \cup \{+\infty\}$ be the maximal time such that, for any $I \in \llbracket 1, K \rrbracket$, any $Z \in \mathbb{Z}$ and all $ (t,x) \in [0,T[ \times \R^3$,
\begin{align}\label{eq:BA1}
\Big| \nabla_{t,x} A_I- \big[ \nabla_{t,x} A_I \big]_{\mathrm{data}} \Big| (t,x) & \leq \frac{C_{\mathrm{boot}} \varepsilon }{\langle t+|x| \rangle\, \langle t-|x| \rangle} ,  \tag{BA1}\\
 \Big| \nabla_{t,x} Z A_I- \big[ \nabla_{t,x} Z A_I \big]_{\mathrm{data}}  \Big| (t,x) & \leq \frac{C_{\mathrm{boot}} \varepsilon }{\langle t+|x| \rangle \, \langle t-|x| \rangle^{1-\eta}}, \tag{BA2} \label{eq:BA2}
\end{align}
for some constant $C_{\mathrm{boot}} >0$ which will be fixed sufficiently large later. By standard local well-posedness argument, $T>0$. Let us improve, if $\varepsilon$ is small enough, the bootstrap assumptions \eqref{eq:BA1}--\eqref{eq:BA2}. This will imply $T=+\infty$.

As an immediate consequence of the bootstrap assumptions, one can estimate the quantities appearing in the commuted Vlasov equation in Proposition \ref{ProCom} and Lemma \ref{LemComdv}. 
\begin{Lem}\label{Lemsource}
If $C_{\mathrm{boot}} \varepsilon \leq \Lambda$, there exists an absolute constant $C>0$ such that, for all $(t,x,v) \in [0,T[ \times \R^3_x \times \R^3_v$, we have
\begin{equation*}
\begin{aligned}
\big| \nabla \mathbf{A} \big|(t,x,v)+\big| \langle v \rangle \nabla_v \nabla \mathbf{A} \big|(t,x,v) & \leq C \Lambda \langle t+|x|\rangle^{-1} \langle t-|x|\rangle^{-1} \langle v \rangle^\delta , \\
\big|\partial_{x^\nu} \nabla \mathbf{A} \big|(t,x,v) & \leq C\Lambda \langle t+|x|\rangle^{-1} \langle t-|x|\rangle^{-7/4} \langle v \rangle^\delta , \\
|\nabla\mathbf{A}_Z|(t,x,v) & \leq C \Lambda \langle t+|x|\rangle^{-1} \langle t-|x| \rangle^{-1+\eta} \langle v \rangle^\delta,
\end{aligned}
\end{equation*}
where $0 \leq \nu \leq 3$, $Z \in \mathbb{K} \setminus \{\partial_{x^0},\dots,\partial_{x^3}\}$.
\end{Lem}
\begin{Rq}\label{Rqconst}
   From now on, except in Proposition \ref{Prozmod} and in the proof of Corollary \ref{Cordecayintv}, all constants in the estimates will be independent of $C_{\mathrm{boot}}$.
\end{Rq}
\begin{proof}
 Let $\Gamma \in \mathbb{K}$. Start by estimating $\big[ \nabla_{t,x} A_I \big]_{\mathrm{data}}$ and $\big[ \nabla_{t,x} \Gamma A_I \big]_{\mathrm{data}}$ using the assumptions satisfied by the initial data (see \eqref{eq:assumpA}--\eqref{eq:assumpB}), and Proposition \ref{Prodata}. Using also \eqref{eq:BA1}--\eqref{eq:BA2}, it yields
    \begin{equation}\label{equa:fortheproof} 
    \big| \nabla_{t,x}  A_I \big| (t,x)+\langle t-|x| \rangle^{-\eta}  \big| \nabla_{t,x} \Gamma A_I \big| (t,x)  \leq \frac{C_{\mathrm{boot}} \varepsilon }{\langle t+|x| \rangle \, \langle t-|x| \rangle}+ \frac{C_0 \Lambda }{ \langle t+|x| \rangle \, \langle t-|x| \rangle^2}  , 
    \end{equation}
where $C_0$ is an absolute constant. In view of the definitions of $\nabla \mathbf{A}$ and $\nabla \mathbf{A}_Z$ given in \eqref{defA} and Proposition \ref{ProCom}, it yields the estimate for these quantities. To deal with $\partial_{x^\mu} \nabla \mathbf{A}$, we combine \eqref{equa:fortheproof} with the classical identity
    \begin{equation}\label{eq:extragain}
    \big| \nabla_{t,x} \Phi \big|(t,x) \lesssim \frac{1}{\langle t - |x| \rangle} \sum_{Z \in \mathbb{Z}} \big| Z \Phi \big|(t,x),
    \end{equation}
    which follows from $(t^2-|x|^2)\partial_t = t S - \sum_j \Omega_{0j}$ and $(t^2-|x|^2)\partial_{x^i}= t \Omega_{0i}-x^iS-\sum_{j}x^j \Omega_{ij}$.
\end{proof}

 \subsection{Estimates on the characteristics}

Before controlling the derivatives of $f$ along the flow, we first establish some estimates for the characteristics $\mathbf{T}_{\mathbf{A}}$.

\begin{Def}
Let, for all $(\tau,x,v) \in [0,T[ \times \R^3_x \times \R^3_v$, $t \mapsto \big( \mathcal X,\mathcal V \big)(t,\tau,x,v)$ be the characteristics associated with the transport operator $\mathbf{T}_{\mathbf{A}}$,
\[ \dot{\mathcal X}=\widehat{\mathcal V}, \qquad \qquad \qquad \dot{\mathcal{V}}=\nabla\mathbf{A}(t,\mathcal X). \]
\end{Def}

We recall here Duhamel's principle. For any distribution function $h$ and all $(x,v) \in \R^3_x \times \R^3_v$, we have
\begin{equation}\label{eq:Duhamel} \forall \, t \in [0,T[ , \qquad \qquad  h(t,x,v) = h \big( 0,\mathcal X(0,t,x,v),\mathcal V(0,t,x,v) \big) + \int_{\tau=0}^t
\big[\mathbf{T}_{\mathbf{A}}(h) \big] \big(\tau ,\mathcal{X}(\tau,t,x,v),\mathcal{V}(\tau,t,x,v) \big) \dr \tau. 
\end{equation}
The following result will then be central in our analysis. 

\begin{Lem}\label{LemforDuhamel}
Let $a >0$, $b >-1$ and
\[ \mathbf{I}_{x,v}^a(t) \coloneqq  \int_{\tau=0}^t \frac{\dr \tau}{ (1+\tau)^{1+a}}, \qquad \mathbf{J}^b_{x,v}(t) \coloneqq \int_{\tau=0}^t \frac{\widehat{\mathcal{V}}^{\underline{L}}(\tau,t,x,v)}{(1 + | \tau - |\mathcal{X}(\tau,t,x,v)||)^{1+b}}  \dr \tau. \]
Then, for all $(t,x,v) \in [0,T[ \times \R^3_x \times \R^3_v$,
\[ \mathbf{I}_{x,v}^a(t) \leq \frac{1}{a}, \qquad \quad \; \mathbf{J}_{x,v}^b(t) \leq \frac{1}{|b|} \quad \text{if $b>0$}, \qquad \quad \; \mathbf{J}^b_{x,v}(t) \leq \frac{2}{|b|}  (1+t)^{|b|} \quad \text{if $b<0$}, \qquad \quad \; \mathbf{J}^1_{x,v}(t) \leq 2 \log (1+ t). \]
\end{Lem}
\begin{proof}
  The bound on $\mathbf{I}_{x,v}^a(t)$ is straightforward. For $\mathbf{J}^b_{x,v}(t)$, we perform a change of variables reflecting that the Vlasov operator reads, in the coordinate system $(u=t-|x|,x,v)$, $\T_{\mathbf{A}}=2\widehat{v}^{\underline{L}}\partial_u+\widehat{v}\cdot\nabla_x+\nabla\mathbf{A}\cdot\nabla_v$. More precisely, let $u(\tau)=\tau -|\mathcal{X}(\tau,t,x,v)|$, so that $u'(\tau)=2 \widehat{\mathcal{V}}^{\underline{L}}(\tau,t,x,v)>0$. It yields, if $b>0$,
\[ \mathbf{J}^b_{x,v}(t) = \frac{1}{2}\int_{u=-|\mathcal{X}(0,t,x,v)|}^{t-|x|} \frac{\dr u}{(1+|u|)^{1+b}} \dr \tau \leq \frac{1}{b}. \]  
For the case $b\leq 0$, note that $u(\tau)=\tau -|\mathcal{X}(\tau,t,x,v)| \leq \tau$. The issue however, to obtain a bound uniform in $(x,v)$, is that $-u(\tau)$ can be arbitrarily large. We then proceed as follows,
\begin{align*}
    \mathbf{J}_{x,v}^b(t) &\leq \int_{\tau=0}^t \frac{ \widehat{\mathcal{V}}^{\underline{L}}(\tau,t,x,v) }{ (1+| \tau-|\mathcal{X}(\tau,t,x,v)||)^{1+b}} \mathds{1}_{u(\tau) \geq - \tau} \dr \tau + \int_{\tau=0}^t \frac{ \widehat{\mathcal{V}}^{\underline{L}}(\tau,t,x,v) }{ (1+| \tau-|\mathcal{X}(\tau,t,x,v)||)^{1+b}} \mathds{1}_{u(\tau) \leq - \tau} \dr \tau  
    \\ & \leq \frac{1}{2} \int_{u=-t}^t \frac{\dr u}{(1+ |u|) \rangle^{1+b}}+\int_{\tau=0}^t \frac{\dr \tau}{ (1+\tau)^{1+b}} , 
    \end{align*}
    which allows to conclude the proof, by considering the cases $b=0$ and $b<0$.
\end{proof}

Let us show that in the case $\delta=0$, velocity characteristics does not deviate too much from the constant in time linear ones.

\begin{Pro}\label{ProVcharac}
If $\delta =0$, there exists a constant $C[\Lambda] >0$ such that, for all $(t,s,x,v) \in [0,T[^2 \times \R^3_x \times \R^3_v$,
\[ \langle \mathcal V(t,s,x,v)\rangle\le 2\langle v\rangle+C[\Lambda]. \]
\end{Pro}
\begin{proof}
Note first, using Lemma \ref{Lemsource}, that
\[ \big| \mathbf{T}_{\mathbf{A}} (\langle v\rangle) \big|= \big| \nabla\mathbf{A} \cdot \nabla_v(\langle v\rangle) \big| \lesssim \Lambda  \langle t+|x|\rangle^{-1}\langle t-|x|\rangle^{-1}. \]
By Lemma \ref{Lemweight}, we have $1 \lesssim \langle v\rangle^{1/2}|\widehat{v}^{\underline{L}}|^{1/4}$, so that
\[ \big|\mathbf{T}_{\mathbf{A}}(\langle v\rangle) \big| \lesssim \frac{\Lambda \langle v\rangle^{1/2}}{\langle t+|x|\rangle^{4/3}} + \frac{\Lambda \langle v\rangle^{1/2}\widehat{v}^{\underline{L}}}{\langle t-|x|\rangle^4}. \]
Use now Duhamel's principle, with initial time $s$ and $h(t,x,v)= \langle v \rangle$, together with the previous Lemma \ref{LemforDuhamel}. We then deduce that for all $t, \, s \in [0,T[$,
\[ \big\langle \mathcal{V}(t,s,x,v) \big\rangle-\langle v\rangle \lesssim \Lambda \big( \mathbf{I}^{\frac{4}{3}}_{x,v}(t)+ \mathbf{J}^4_{x,v}(t) \big) \sup_{0 \leq \tau < T} \big\langle \mathcal{V}(\tau,s,x,v)\rangle^{1/2} \leq 4a^{-1}\Lambda + 4a \Lambda \sup_{ 0 \leq \tau <T} \big\langle \mathcal{V}(\tau,s,x,v) \big\rangle, \]
for any $a>0$. Choosing $a$ small enough yields the result.
\end{proof}

We now treat the case $\delta=1$, for which the deviation of the momentum nonlinear characteristics from the linear ones can be worse. 

\begin{Pro}\label{ProVcharacbis}
If $\Lambda$ is small enough, there exists an absolute constant $C_1>0$ such that, for all $(s,t,x,v) \in [0,T[^2 \times \R^3_x \times \R^3_v$,
\[ \big\langle \mathcal V(s,t,x,v) \big\rangle \leq \frac{C_1}{\kappa}  \big\langle \mathcal V(0,t,x,v) \big\rangle^{1+\kappa},  \]
for $9C\Lambda < \kappa <1/2$, where $C>0$ is the absolute constant introduced in Lemma \ref{Lemsource}.
\end{Pro}
\begin{proof}
According to Lemma \ref{Lemsource}, there exists $C>0$ such that
\[ \big| \mathbf{T}_{\mathbf{A}} (\langle v\rangle) \big|= \big| \nabla\mathbf{A} \cdot \nabla_v(\langle v\rangle) \big| \leq C\Lambda \langle v \rangle \langle t+|x|\rangle^{-1}\langle t-|x|\rangle^{-1}. \]
As $\langle t+|x| \rangle^{-1} \leq \sqrt{2}(1+t)^{-1}$, Duhamel's principle \eqref{eq:Duhamel} and an application of Grönwall's inequality yield
\[ \big\langle \mathcal{V}(t,s,x,v) \big\rangle \leq \langle \mathcal{V}(0,s,x,v) \rangle e^{ \sqrt{2} C \Lambda \log (1+t)}.  \]
To improve this estimate, we write, for $0<a <1$,
\[ \big| \mathbf{T}_{\mathbf{A}} (\langle v\rangle) \big| \lesssim \frac{\Lambda \langle v \rangle^{1+2a} |\widehat{v}^{\underline{L}}|^a }{ \langle t+|x|\rangle \,\langle t-|x|\rangle} \leq  \frac{\Lambda \langle v \rangle^{1+2a}  }{ \langle t+|x|\rangle^{1+a}}+\frac{\Lambda \langle v \rangle^{1+2a} |\widehat{v}^{\underline{L}}| }{ \langle t+|x|\rangle^a \langle t-|x|\rangle^{\frac{1}{a}}}. \]
Note now that if $\Lambda$ is small enough compared to $a$, $\langle \mathcal{V}(t,s,x,v) \big\rangle \leq \langle \mathcal{V}(0,s,x,v) \rangle (1+t)^{\frac{a}{2(1+2a)}}$. Hence, By Duhamel's principle \eqref{eq:Duhamel}, we have
\[ \big\langle \mathcal{V}(t,s,x,v) \big\rangle - \big\langle \mathcal{V}(0,s,x,v) \big\rangle  \lesssim \big\langle \mathcal{V}(0,s,x,v) \big\rangle^{1+2a}\int_{\tau=0}^t \frac{\Lambda }{\langle \tau \rangle^{1+\frac{a}{2}}}+ \frac{\Lambda  \widehat{\mathcal{V}}^{\underline{L}}(\tau,s,x,v)}{\langle \tau -|\mathcal{X}(\tau , s,x,v)|\rangle^{\frac{1}{a}}} \dr \tau . \]
It remains to use Lemma \ref{LemforDuhamel}.
\end{proof}

If $\delta=1$, we assume throughout Sections \ref{SecAsympstab}--\ref{SEcscatt} that $9C\Lambda<1/2$, and we fix a constant $\kappa$ such that $9C\Lambda<\kappa<1/2$. We conclude this section by deriving qualitative properties of the function $s \mapsto \tau - |\mathcal{X}(s,t,x,v)| $. 

\begin{Pro}\label{Prodomaincharac}
Let $(t,x,v) \in \R_+ \times \R^3_x \times \R^3_v$. Then, the following holds.
\begin{itemize}
\item \label{it:A} There exists at most one time $0 \leq t_0 < T$ such that $t_\star-|\mathcal{X}(t_\star,t,x,v) |=0$. We set $t_\star \coloneqq t_0$ if $t_0$ is exists and $t_\star \coloneqq T$ otherwise. 
\item There exists $C[\Lambda,\kappa]>0$ such that, for all $0 \leq t <T$,
\[ \forall \, 0 \leq s \leq t_\star, \quad \big| \mathcal{X}(s,t,x,v) \big|-s \geq a (t_\star-s), \qquad \qquad \forall \, t_\star \leq s < T, \quad s- \big|\mathcal{X}(s,t,x,v) \big| \geq a (s-t_\star), \]
where $a \coloneqq C \, \langle \mathcal{V}(0,t,x,v) \rangle^{-2-2\delta \kappa}$ satisfy $0<a<1$.
\item There exists $c[\Lambda, \kappa] >0$ such that $ t_\star \leq c \big| \mathcal{X}(0,t,x,v) \big| \, \langle \mathcal{V} (0,t,x,v) \rangle^{2+2\delta \kappa}$.
\end{itemize}
\end{Pro}
\begin{proof}
  Let us denote $(\mathcal{X},\mathcal{V})(s,t,x,v)$ by $(\mathcal{X}_s,\mathcal{V}_s)$. The starting point consists in noticing that
  \begin{equation*}
  \frac{\dr }{\dr s} \big( s-|\mathcal{X}_s| \big)=1-\frac{\mathcal{V}_s}{\langle \mathcal{V}_s \rangle } \cdot \frac{\mathcal{X}_s}{|\mathcal{X}_s|} \geq \frac{1}{\langle \mathcal{V}_s \rangle^2} \gtrsim \frac{1}{\langle \mathcal{V}(0,t,x,v) \rangle^{2+2\delta \kappa}}, 
  \end{equation*}
  where, in the last step, we used Propositions \ref{ProVcharac}--\ref{ProVcharacbis}. This provides the uniqueness of $t_\star$, as well as it existence if $T=+\infty$. We also obtain from this lower bound the estimate on $|s-|\mathcal{X}_s||$. For the upper bound on $t_\star$, note that
  \[ \big| \mathcal{X}_s \big| \leq \big| \mathcal{X}_0 \big|+s \sup_{0 \leq \tau \leq s } \big| \widehat{\mathcal{V}}_\tau \big|.  \]
  Moreover, as $z \in \R_+ \mapsto z/\langle z \rangle$ increases, we have from Propositions \ref{ProVcharac}--\ref{ProVcharacbis} that there exists $c[\Lambda,\kappa]>0$ such that
  \[ \sup_{0 \leq \tau \leq s } \big| \widehat{\mathcal{V}}_\tau \big| \leq \frac{c \, \langle \mathcal{V}_0 \rangle^{1+\delta\kappa}}{\sqrt{1+c^2 \langle \mathcal{V}_0 \rangle^{2+2\delta\kappa}}} .  \]
  We then deduce that
  \[  \big| \mathcal{X}_0 \big|+s \sup_{0 \leq \tau \leq s } \big| \widehat{\mathcal{V}}_\tau \big| >s  \Longrightarrow s < 2c^2 \big| \mathcal{X}_0 \big|   \langle \mathcal{V}_0 \rangle^{2+2\delta\kappa}, \]
  which implies $t_\star \leq 2c^2 \big| \mathcal{X}_0 \big|   \langle \mathcal{V}_0 \rangle^{2+2\delta\kappa}$. 
\end{proof}

\subsection{Pointwise estimates on the first order derivatives of $f$}\label{Subsec32}

We start by controlling $\partial_{t,x} f$ and $\langle v \rangle \partial_{v}f$ using the strategy outlined in Section \ref{SubsecKeyidea}. The commutation formula of Proposition \ref{ProCom} and Lemma \ref{LemComdv}, together with the decay estimates in Lemma \ref{Lemsource}, provide
\begin{equation}\label{eq:latoutdesuite0}
\begin{aligned}
\big| \mathbf{T}_{\mathbf{A}} \big( |\partial_{x^\mu}f| \big) \big| & \leq \frac{2C \Lambda \langle v\rangle^{\delta -1 } }{(1+t) \, \langle t-|x|\rangle^{\frac{7}{4}}}  \cdot \langle v \rangle \big| \nabla_{v}f \big| (t,x,v) , \\
 \big| \mathbf{T}_{\mathbf{A}} \big( \langle v \rangle |\partial_{v^i} f| \big) \big| &  \leq  
\frac{4C\Lambda \langle v \rangle^{\delta -1 } }{ (1+t) \, \langle t-|x|\rangle}   \cdot \langle v \rangle \big| \nabla_{v} f \big|(t,x,v) +  \big|\nabla_x f \big|(t,x,v),
\end{aligned}
\end{equation}
for any $0 \leq \mu \leq 3$ and $1 \leq i \leq 3$. The next result will be used to estimate the derivatives of $f$ and $\langle x -t \widehat{v} \rangle$ along the flow. It is for the latter case that we consider source terms $S_1$ and $S_2$.

\begin{Lem}\label{LemDuhamel}
Let $\mathbf{c}>0$ and $g_1, \, g_2, \, S_1, \, S_2 \colon [0,T[ \times \R^3_x \times \R^3_v \to \R_+$ be functions such that
\begin{equation*}
    \begin{aligned}
\big| \mathbf{T}_{\mathbf{A}} ( g_1 ) \big| & \leq \frac{\mathbf{c} \, \langle v\rangle^{\delta -1 } }{(1+t) \, \langle t-|x|\rangle^{\frac{3}{2}}}  \cdot g_2 +S_1 , \\
 \big| \mathbf{T}_{\mathbf{A}} ( g_2 ) \big| &  \leq  
\frac{\mathbf{c} \, \langle v \rangle^{\delta -1 } }{ (1+t) \, \langle t-|x|\rangle}   \cdot g_2 +  g_1+S_2,
\end{aligned}
\end{equation*}
where the source terms satisfy, for $\mathbf{B}_0>0$ and for all $(t,x,v) \in [0,T[ \times \R^3_x \times \R^3_v$,
\[ \big| S_1 \big|(t,x,v) \leq \mathbf{B}_0 \frac{\langle v \rangle^{\delta - 1}}{\langle t - |x| \rangle^{\frac{3}{2}}}, \qquad \qquad \big| S_2 \big|(t,x,v) \leq \mathbf{B}_0 . \]
Then, there exists two constants $b[\mathbf{c},\Lambda,\kappa], \, B[\Lambda, \kappa]>0$ such that, for all $(t,x,v) \in [0,T[ \times  \R^3_x \times \R^3_v$,
\begin{align*}
|g_1|(t,x,v)+\frac{|g_2|(t,x,v)}{1+t}  & \leq  b \Big(\big[g_1+g_2 \big]\big(0,\mathcal{X}(0,t,x,v), \mathcal{V}(0,t,x,v) \big)+\mathbf{B}_0 \langle \mathcal{V} (0,\tau , x,v) \rangle^3 \Big) \\
& \quad \, \times \langle \mathcal{X} (0,t,x,v) \rangle^{3} \langle \mathcal{V} (0,\tau , x,v) \rangle^6  e^{B \sqrt{\mathbf{c}} \,  \langle \mathcal{V}(0,\tau , x,v) \rangle^{\frac{1+\delta}{2} +\delta \kappa}}  . 
\end{align*}
\end{Lem}
\begin{proof}
    Let $(t,x,v) \in [0,T[ \times \R^3_x \times \R^3_v$. To lighten the notation, we simply denote $(\mathcal{X},\mathcal{V})(s,t,x,v)$ by $(\mathcal{X}_s,\mathcal{V}_s)$. We will make use of the following bounds on the source terms,
    \begin{equation}\label{eq:boundS1S2}
       \int_{s=0}^T  \big| S_1 \big(s,\mathcal{X}_s, \mathcal{V}_s \big) \big| \dr s  \leq c \mathbf{B}_0 \langle \mathcal{V}_0 \rangle^{1+\delta (1+2\kappa)} , \qquad \qquad
        \int_{s=\tau_1}^{\tau_2}  \big| S_2 \big(s,\mathcal{X}_s, \mathcal{V}_s \big) \big| \dr s  \leq \mathbf{B}_0 (\tau_2-\tau_1), 
    \end{equation}
  where $c[\Lambda, \kappa] >0$. The second estimate follows from $|S_2| \leq \mathbf{B}_0$. For the first one, we use Lemma \ref{Lemweight} to get $1 \leq 4 \langle v \rangle^2 \widehat{v}^{\underline{L}}$ and then Propositions \ref{ProVcharac}--\ref{ProVcharacbis}. It yields
   \[  \big| S_1 \big(s,\mathcal{X}_s, \mathcal{V}_s \big) \big| \lesssim \mathbf{B}_0 \frac{\langle \mathcal{V}_s \rangle^{1+\delta } \widehat{\mathcal{V}}_s^{\underline{L}}}{\langle s- |\mathcal{X}_s| \rangle^{\frac{3}{2}}} \lesssim \mathbf{B}_0 \langle \mathcal{V}_0 \rangle^{1+\delta (1+2\kappa) } \frac{ \widehat{\mathcal{V}}_s^{\underline{L}}}{\langle s- |\mathcal{X}_s| \rangle^{\frac{3}{2}}}. \]
   We conclude by using Lemma \ref{LemforDuhamel}, which provides $\mathbf{J}_{x,v}^{\frac{1}{2}}(t) \leq 2$. Accordingly with \eqref{defcalF}, we define
\begin{equation*}
 \mathcal{G} (t,x,v) \coloneqq  g_1(t,x,v)+ \mathbf{b}(t,x,v) g_2(t,x,v), \qquad \qquad \mathbf{b}(t,x,v) \coloneqq \frac{\sqrt{\mathbf{c}}\, \langle \mathcal{V}_0 \rangle^{\frac{\delta -1 }{2}}}{\sqrt{1+t} \, (1+ | t-|x||)^{\frac{3}{4}}} . 
\end{equation*}
Note now that $\langle v \rangle^{\frac{\delta-1}{2}}=\langle \mathcal{V}_0 \rangle^{\frac{\delta -1}{2}}=1$ if $\delta =1$ and $\langle v \rangle^{-1} \leq C\langle \mathcal{V}_0 \rangle^{-1}$ if $\delta =0$ by Proposition \ref{ProVcharac}, where $C[\Lambda ] \geq 4$. We then deduce that
\[  \T_{\mathbf{A}} \big( \mathcal{G} \big)  \leq (4 C+1) \mathbf{b} \mathcal{G}+\max \big( \T_{\mathbf{A}}(\mathbf{b}),0 \big)  \mathcal{G} +S_1+\mathbf{b}S_2 . \]
We use Duhamel's principle \eqref{eq:Duhamel} and Grönwall's inequality to get, for all $0 \leq t < T$,
\begin{equation}\label{eq:esticalG}
g_1 \big( t, \mathcal{X}_t, \mathcal{V}_t \big) \leq  \mathcal{G} \big(t, \mathcal{X}_t, \mathcal{V}_t  \big) \leq  \big[ \mathcal{G} \big(0, \mathcal{X}_0, \mathcal{V}_0  \big) + c \mathbf{B}_0 \langle \mathcal{V}_0 \rangle^{1+\delta(1+2\kappa)}+\mathbf{B}_0 \mathbf{K}_{x,v}(t) \big]e^{ 5C  \mathbf{K}_{x,v}(t)+ \mathbf{L}_{x,v}(t) },  
\end{equation}
where we have bounded the source terms using \eqref{eq:boundS1S2} and $|S_2| \leq \mathbf{B}_0$, and where
\begin{equation}\label{defKL}
\mathbf{K}_{x,v}(t) \coloneqq \int_{s=0}^t \mathbf{b} \big( s, \mathcal{X}_s, \mathcal{V}_s \big) \dr s, \qquad \qquad   \mathbf{L}_{x,v}(t) \coloneqq \int_{s=0}^t \max \big ( \T_{\mathbf{A}} \big(\log \mathbf{b} \big),0 \big) \big( s, \mathcal{X}_s, \mathcal{V}_s \big)  \dr s  . 
\end{equation}
Let us admit for now that there exists a constant $C_0[\Lambda,\kappa]>0$ such that 
\begin{equation}\label{estiIJ} 
\mathbf{K}_{x,v}(t) \lesssim \sqrt{\mathbf{c}} \, \langle \mathcal{V}_0 \rangle^{\frac{1+\delta}{2}+\delta \kappa}, \qquad \quad   \mathbf{L}_{x,v}(t) \coloneqq \int_{s=t_{\mathrm{in}}}^t \T_{\mathbf{A}} \big(\log \mathbf{b} \big) \big( s, \mathcal{X}_s, \mathcal{V}_s \big)  \dr s \leq C_0+3 \log \langle \mathcal{X}_0 \rangle+6\log \langle \mathcal{V}_0 \rangle . 
\end{equation}
We then deduce, as $\delta(1+2\kappa) \leq 2$, that for all $0 \leq t< T$,
\begin{align*}
    g_1 \big( t, \mathcal{X}_t, \mathcal{V}_t \big) \leq  \mathcal{G} \big(t, \mathcal{X}_t, \mathcal{V}_t  \big) & \leq e^{C_0}\big[ \mathcal{G} \big(0, \mathcal{X}_0, \mathcal{V}_0  \big) + B_1\mathbf{B}_0 \langle \mathcal{V}_0 \rangle^3 \big] \langle \mathcal{X}_0 \rangle^3 \langle \mathcal{V}_0 \rangle^6 e^{ C_1 \sqrt{\mathbf{c}} \,   \langle \mathcal{V}_0 \rangle^{ \frac{1+\delta}{2}+\delta \kappa}}.
    \end{align*}
for some constants $ B_1[\sqrt{\mathbf{c}},\Lambda,\kappa] > 0  $ and $C_1 [\Lambda,\kappa] >0$ depending only on $(\Lambda,\kappa)$. Using now that
\[  \big| \T_{\mathbf{A}} (g_2) \big| \leq \big( \sqrt{2\mathbf{c}}C+1 \big) \mathcal{G}+S_2,\]
we obtain, for all $0 \leq t < T$, that
\begin{align*}
     g_2   \big( t, \mathcal{X}_t, \mathcal{V}_t \big) & \leq g_2 \big(0, \mathcal{X}_0, \mathcal{V}_0 \big)+   t \big( \sqrt{2\mathbf{c}}C+1 \big)  e^{C_0}\big[ \mathcal{G} \big(0, \mathcal{X}_0, \mathcal{V}_0  \big) + B_1\mathbf{B}_0 \langle \mathcal{V}_0 \rangle^3 \big] \langle \mathcal{X}_0 \rangle^3 \langle \mathcal{V}_0 \rangle^6 e^{ C_1 \sqrt{\mathbf{c}} \,   \langle \mathcal{V}_0 \rangle^{ \frac{1+\delta}{2}+\delta \kappa}}+ \mathbf{B}_0 t.
     \end{align*}
It then remains to prove \eqref{estiIJ}. Recall from Proposition \ref{Prodomaincharac} that $|s-|\mathcal{X}_s|| \geq a |s-t_\star|$, where $a = C' \, \langle \mathcal{V}_0 \rangle^{-2-2\delta\kappa}$. We then deduce
\begin{align*}
    \mathbf{K}_{x,v}(t) & \leq \! \int_{s=0}^{t}  \frac{\sqrt{\mathbf{c}} \, \dr s}{\langle \mathcal{V}_0 \rangle^{\frac{1-\delta}{2}}\sqrt{s} (1+| s-|\mathcal{X}_s||)^{\frac{3}{4}}} \leq  \frac{\sqrt{\mathbf{c}} }{\sqrt{a} \langle \mathcal{V}_0 \rangle^{\frac{1-\delta}{2}} }\int_{s=0}^{+\infty} \frac{\dr \tau}{\sqrt{\tau} (1+| \tau-a t_\star |)^{\frac{3}{4}}} \lesssim \frac{ \sqrt{\mathbf{c}} }{\sqrt{a} \langle \mathcal{V}_0 \rangle^{\frac{1-\delta}{2}} } \lesssim \sqrt{\mathbf{c}}  \, \langle \mathcal{V}_0 \rangle^{\frac{1+\delta}{2}+\delta \kappa}.
\end{align*}
For $\mathbf{L}_{x,v}(t)$, we use $\T_{\mathbf{A}}(t-|x|)=2 \widehat{v}^{\underline{L}}$ to obtain
\[\T_{\mathbf{A}}(\mathbf{b}) = -\frac{1}{2(1+t)} \mathbf{b}- \frac{6 \widehat{v}^{\underline{L}} }{4 \, (1+ |t-|x||) } \mathrm{sgn}(t-|x|) \mathbf{b} \leq   \frac{ 3\widehat{v}^{\underline{L}} }{ 2(1+|x|-t)} \mathds{1}_{|x| >t} \mathbf{b}   , \] 
where $\mathrm{sgn}(y)$ denotes the sign of $y$. We then deduce that $\T_{\mathbf{A}}(\log \mathbf{b} )(s,\mathcal{X}_s,\mathcal{V}_s) \leq 0$ for $s \geq t_\star$ and
\[   \mathbf{L}_{x,v}(t) \leq \frac{3}{2}\int_{s=0}^{\min (t,t_\star)} \frac{ \widehat{\mathcal{V}}_s^{\underline{L}} }{ 1+|\mathcal{X}_s| -1} = \frac{3}{2} \mathbf{J}_{x,v}^1 \big( \min (t,t_\star) \big) \leq 3 \log (1+t_\star),   \]
where, in the last step, we used Lemma \ref{LemforDuhamel}. Finally, apply Proposition \ref{Prodomaincharac} to obtain $t_\star \lesssim \langle \mathcal{X}_0 \rangle \, \langle \mathcal{V}_0 \rangle^2$.
\end{proof}

Let us now control the first order derivatives of $f$.

\begin{Pro}\label{Proboundedness}
There exists a constant $\mathbf{B}[\Lambda,\kappa] >0$ such that, for all $(t,x,v) \in [0,T[ \times \R^3_x \times \R^3_v$,
\[ \big|\nabla_{t,x} f \big|(t,x,v)+ \sup_{ \widehat{Z} \in \Pp_S } \langle t\rangle^{-\eta} \big|\widehat{Z} f \big|(t,x,v) \lesssim \langle \mathcal{X} (0,t,x,v) \rangle^3 \langle \mathcal{V} (0,t,x,v) \rangle^9 e^{\mathbf{B} \sqrt{\Lambda} \,   \langle \mathcal{V}(0,\tau ,x,v) \rangle^{\frac{1+\delta}{2}+\delta \kappa}} \mathcal{M}(t,x,v) ,  \]
where $\mathcal{M}$ is determined by the initial data and is given by
\[ \mathcal{M}(t,x,v) \coloneqq \big[ \langle x \rangle\big|\nabla_{t,x} f \big|(t,x,v)+ \langle v \rangle  \big| \nabla_v f \big| \big] \big(0,\mathcal{X} (0,t,x,v),\mathcal{V} (0,t,x,v) \big). \]
\end{Pro}
\begin{proof}
Note first that \eqref{eq:latoutdesuite0} and the previous Lemma \ref{LemDuhamel}, applied with $S_1=S_2=0$, $g_1=\sum_\mu|\partial_{x^\mu}g|$ and $g_2=\sum_i \langle v \rangle|\partial_{v^i} f|$, yield
\begin{equation}\label{eq:pointestiforproof}  \big|\nabla_{t,x} f \big|(t,x,v)+\frac{\langle v \rangle}{1+t} \big| \nabla_v f \big| (t,x,v) \lesssim  \langle \mathcal{X} (0,t,x,v) \rangle^3 \langle \mathcal{V} (0,t,x,v) \rangle^6 e^{B  \sqrt{\Lambda} \,   \langle \mathcal{V}(0,\tau , x,v) \rangle^{\frac{1+\delta}{2} +\delta \kappa}} \mathcal{M}(t,x,v), 
\end{equation}
for some constant $B[\Lambda,\kappa] >0$. Note now the relation
\[ \langle v\rangle\partial_{v^i}=\widehat{\Omega}_{0i}-t\partial_{x^i}-x^i\partial_t. \]
Together with the commutation formula of Proposition \ref{ProCom}, the decay estimates in Lemma \ref{Lemsource} and $1 \lesssim \langle v \rangle^2 \widehat{v}^{\underline{L}}$, we obtain
\begin{align*}
 \big|\mathbf{T}_{\mathbf{A}} \big( |\widehat{Z} f| \big) \big| &  \leq C\Lambda \frac{\langle v \rangle^{1+\delta} \widehat{v}^{\underline{L}}}{\langle t-|x|\rangle^{1-\eta}} \big| \nabla_{t,x} f \big|(t,x,v) + \frac{C\Lambda \langle v\rangle^{\delta} |\widehat{v}^{\underline{L}}|^{\frac{1}{2}} }{\langle t+|x|\rangle^{1-\eta} \langle t-|x|\rangle } \sup_{\widehat{\Gamma} \in \Pp_S} \big|\widehat{\Gamma} f \big|(t,x,v), 
\end{align*}
for any $\widehat{Z} \in \Pp_S$ and for some constant $C>0$. Let $(t,x,v) \in [0,T [ \times \R^3_x \times \R^3_v$ and let us denote $(\mathcal{X},\mathcal{V})(\tau , t , x ,v)$ by $(\mathcal{X}_\tau , \mathcal{V}_\tau)$. Duhamel's principle \eqref{eq:Duhamel} and the Grönwall Lemma provide
\[ \sup_{\widehat{Z} \in \Pp_S} \big| \widehat{Z}f \big|(t,x,v) \leq \bigg( \sup_{\widehat{Z} \in \Pp_S} \big| \widehat{Z}f \big| \big(0,\mathcal{X}_0,\mathcal{V}_0 \big) + \sqrt{2} C \Lambda \mathbf{J}^{-\eta}_{x,v}(t)\sup_{0 \leq \tau \leq t} \langle \mathcal{V}_\tau \rangle^{1+\delta} \big| \nabla_{t,x} f \big| \big( \tau , \mathcal{X}_\tau , \mathcal{V}_\tau \big) \bigg) e^{C \Lambda \mathbf{Q}_{x,v}(t)} , \]
where, using the notation and the bounds in Lemma \ref{LemforDuhamel}, 
\[ \mathbf{J}_{x,v}^{- \eta }(t) \leq  \frac{2}{\eta}  (1+t)^\eta ,  \qquad \qquad \mathbf{Q}_{x,v}(t) \coloneqq \int_{s=0}^t \frac{\langle \mathcal{V}_s \rangle^{\delta} | \widehat{\mathcal{V}}_s |^{\frac{1}{2}}}{\langle s \rangle^{1-\eta} \langle s - |\mathcal{X}_s | \rangle} \dr s .  \]
According to Proposition \ref{ProVcharacbis}, we have $\langle \mathcal{V}_\tau \rangle^{\delta } \lesssim \langle \mathcal{V}_0 \rangle^{\delta+ \delta \kappa}$. Thus, Young's inequality and Lemma \ref{LemforDuhamel} provide
\begin{equation}\label{eq:estiQ}
\mathbf{Q}_{x,v}(t) \lesssim \langle \mathcal{V}_0 \rangle^{\delta+ \delta\kappa} \big( \mathbf{I}_{x,v}^{1-2\eta}(t) +\mathbf{J}^1_{x,v}(t) \big) \leq   3\langle \mathcal{V}_0 \rangle^{\delta+ \delta\kappa}. 
\end{equation}
We conclude the proof by bounding $ \langle \mathcal{V}_\tau \rangle^{1+\delta} \big| \nabla_{t,x} f \big| \big( \tau , \mathcal{X}_\tau , \mathcal{V}_\tau \big)$ through \eqref{eq:pointestiforproof} and Propositions \ref{ProVcharac}--\ref{ProVcharacbis}, which provide $\langle \mathcal{V}_\tau \rangle^{1+\delta} \lesssim \langle \mathcal{V}_0 \rangle^{1+\delta(1+\kappa)} \leq \langle \mathcal{V}_0 \rangle^3$, and by choosing $\mathbf{B}$ large enough.
\end{proof}

Since the initial distribution function satisfies either \eqref{eq:assumpA} or \eqref{eq:assumpB}, we deduce the following estimates from the previous Proposition \ref{Proboundedness}.

\begin{Cor}\label{Corboundedness}
There exists a constant $\mathbf{D}[\Lambda,\kappa]>0$ such that, for all $(t,x,v) \in [0,T[ \times \R^3_x \times \R^3_v$,
\[ \big\langle \mathcal{X} (0,t,x,v) \big\rangle^{8} \langle \mathcal{V}(0,t,x,v) \rangle^{149}  \bigg( |f| (t,x,v)+  \big|\nabla_{t,x} f \big|(t,x,v)+ \sup_{ \widehat{Z} \in \Pp_S } \langle t\rangle^{-\eta} \big|\widehat{Z} f \big|(t,x,v) \bigg) \leq \mathbf{D} \varepsilon. \]
\end{Cor}

\subsection{Almost optimal decay estimates for velocity averages}

To close the bootstrap argument, we follow a two-step strategy designed to address the difficulty discussed in Section \ref{Subsubsecmachin}.
\begin{itemize}
    \item First, we obtain almost optimal decay rates as a direct consequence of the results established in Section \ref{Subsec32}. It allows to refine our control of the electromagnetic field.
    \item This, in turn, enables a more precise description of the behavior of $\mathcal{X}(0,t,x,v)$ and yields suitable estimates for the momentum averages, allowing us to close the bootstrap argument.
\end{itemize}
We start by showing a rough estimate of $\mathcal{X}(0,t,x,v)$.

\begin{Pro}\label{Prozmod}
There exist $\mathbf{b}[\Lambda, \kappa] >0$ such that, for all $(t,x,v) \in [0,T[ \times \R^3_x \times \R^3_v$,
\[  \big|x-t\widehat{v}-\mathcal{X}(0,t,x,v) \big| \leq \mathbf{b} \varepsilon  \big\langle \mathcal{V}(0,t,x,v) \big\rangle^{1+\delta(1+2\kappa)} \log \, \langle t+1 \rangle. \]
\end{Pro}
\begin{proof}
Recall from Lemma \ref{Lemweight} that $1\lesssim\langle v\rangle^2\widehat{v}^{\underline{L}}$. We then obtain from Lemma \ref{Lemsource} that
\[ \big|\T_{\mathbf{A}} \big( x-t\widehat{v} \big) \big| = \big| -t \nabla \mathbf{A} \cdot \nabla_v \big( \widehat{v} \big) \big| \lesssim  \frac{\Lambda \, \langle v\rangle^{1+\delta} \widehat{v}^{\underline{L}}}{ \langle t-|x|\rangle} . \]
We then deduce from Duhamel's principle \eqref{eq:Duhamel} and Propositions \ref{ProVcharac}--\ref{ProVcharacbis} that
\[ \big|x-t\widehat{v}-\mathcal{X}(0,t,x,v) \big| \lesssim \Lambda \big\langle \mathcal{V}(0,t,x,v) \big\rangle^{1+\delta(1+2\kappa)} \mathbf{J}_{x,v}^0(t).  \]
The result follows from Lemma \ref{LemforDuhamel}.
\end{proof}

Combining the previous Proposition \ref{Prozmod} with Propositions \ref{ProVcharac}--\ref{ProVcharacbis} and Corollary \ref{Corboundedness}, we obtain the next result.
\begin{Cor}\label{Corboundednessbof}
There exists a constant $\mathbf{D}'[\Lambda,\kappa]>0$ such that, for all $(t,x,v) \in [0,T[ \times \R^3_x \times \R^3_v$,
\[ \langle x-t\widehat{v} \rangle^6 \langle v \rangle^{18}  \bigg( |f| (t,x,v)+  \big|\nabla_{t,x} f \big|(t,x,v)+ \sup_{ \widehat{Z} \in \Pp_S } \langle t\rangle^{-\eta} \big|\widehat{Z} f \big|(t,x,v) \bigg) \leq \mathbf{D}' \varepsilon \log^6 \langle t+1 \rangle. \]
\end{Cor}

This allows us to derive almost optimal decay estimates for momentum averages.

\begin{Cor}\label{Cordecayintvbof}
 We have, for any $\widehat{Z} \in \mathbb{P}_S$ and all $(t,x) \in [0,T[ \times \R^3_x$,
   \[ \int_{\R^3_v} \langle v \rangle^6 |f| (t,x,v) \dr v \lesssim \varepsilon\frac{\log^6 \, \langle t+1 \rangle}{\langle t+ |x| \rangle^3}, \qquad \qquad \int_{\R^3_v} \langle v \rangle^6 \big| \widehat{Z} f \big| (t,x,v) \dr v \lesssim \varepsilon\frac{\log^6 \, \langle t+1 \rangle}{\langle t+ |x| \rangle^{3-\eta}}. \]
   If $|x| \geq t$, we have the improved estimate
   \[ \int_{\R^3_v} \langle v \rangle^6 |f| (t,x,v) \dr v+ \int_{\R^3_v} \langle v \rangle^6 \big| \widehat{Z} f \big| (t,x,v) \dr v \lesssim \frac{ \varepsilon}{\langle t+ |x| \rangle^{\frac{11}{4}}}. \]
\end{Cor}
\begin{proof}
Observe first, with $h=f$ or $h=\widehat{Z}f$, that
\begin{align*}
    \int_{\R^3_v} \langle v \rangle^6 \big|h \big|(t,x,v) \dr v & \leq \int_{\R^3_v} \frac{\dr v}{\langle x-t \widehat{v}\rangle^4 \langle v \rangle^{12}} \sup_{w \in \R^3_v} \; \langle x-t \widehat{w}\rangle^4 \langle w \rangle^{18} \big|  h(t,x,w) \big|.
\end{align*}
The second factor on the right hand side can be controlled using the previous Corollary \ref{Corboundednessbof}. Finally, we obtain
\[ \int_{\R^3_v} \frac{\dr v}{\langle x-t \widehat{v}\rangle^4 \langle v \rangle^{12}} \lesssim \frac{1}{\langle t \rangle^3} \mathds{1}_{t > |x|}+ \frac{1}{\langle t \rangle^4} \mathds{1}_{ |x| \geq t} \]
by using Lemma \ref{Lemcdv} to perform the change of variables $y=x-t\widehat{v}$ when $t > |x|$, while, for $|x| \geq t$, we use
  \[  |x-t\widehat{v}| \geq |x| (1-|\widehat{v}|) =|x| \frac{1}{\langle v \rangle (\langle v \rangle +|v| )} \geq \frac{|x|}{2 \langle v \rangle^2}.\]
\end{proof}

In the perspective of improving the bootstrap assumptions \eqref{eq:BA1}--\eqref{eq:BA2}, we show here that one of the two dynamical terms in the Glassey-Strauss decomposition of $\partial_{x^\mu} A_I$ and $\partial_{x^\mu} Z A_I$ enjoys strong decay along timelike straight lines. We will make use of the following result. For a proof, we refer to \cite[Lemmata~5.10--5.11]{scat}.

\begin{Lem}\label{Lemint}
Let $a >3 $. Then, for all $(t,x) \in \R_+ \times \R^3$, we have
    $$ \mathbf{Y}^{1}_{a}(t,x):= \int_{|y-x| \leq t } \frac{1}{\langle t-|y-x|+|y| \rangle^{a} \langle t-|y-x|-|y| \rangle} \frac{\dr y}{|y-x|}\lesssim \frac{\log(3+|t-|x||)}{\langle t+|x| \rangle \langle t-|x| \rangle^{a-2}}.$$
    For any $b \geq 3$ and all $(t,x) \in \R_+ \times \R^3$, we have
$$ \mathbf{Y}^{2}_b(t,x):= \int_{|y-x| \leq t } \big\langle t-|y-x|+|y| \big\rangle^{-b} \frac{\dr y}{|y-x|^2} \lesssim \frac{1}{ \langle t+ |x| \rangle^{\frac{7}{4}}\langle t-|x| \rangle^{b-3+\frac{1}{4}}}  .$$
\end{Lem}

It allows to obtain the following estimates.

\begin{Pro}\label{ProAIT0}
Let $I \in \llbracket 1 , K \rrbracket$ and $Z \in \mathbb{K}$. Then, for all $(t,x) \in [0,T[ \times \R^3$, we have
 \[ \big| \big[ \nabla_{t,x} A_I \big]_{\T_0} \big|(t,x)+\big| \big[ \nabla_{t,x} Z A_I \big]_{\T_0} \big|(t,x) \lesssim  \frac{\Lambda \varepsilon}{\langle t+|x| \rangle \, \langle t-|x| \rangle^{\frac{7}{4}}}. \]
\end{Pro}
\begin{proof}
Let $p \in \{0, 1\}$, with the convention that $Z^p A_I=A_I$ if $p=0$ and $Z^p A_I= ZA_I$ if $p=1$, and recall from Proposition \ref{ProGS} the expression of $[\nabla_{t,x} Z^p A_I]_{\T_0}$. Then, we use the Vlasov equation to write $\T_0(f)=-\nabla \mathbf{A} \cdot \nabla_v f$, where $\nabla \mathbf{A}$ is given by \eqref{defA}. Using the commutation formula of Proposition \ref{ProCom}, we have a similar relation for $\T_0(\widehat{Z}f)$. Hence, performing integration by parts in $v$, we obtain
\[ \big[\nabla_{t,x} Z^pA_I \big]_{\T_0} \lesssim \sum_{p_1+p_2 \leq p} \;\sum_{ 1 \leq J \leq K} \int_{|y-x|\leq t} \big|\nabla_{t,x} Z^{p_1} A_J(t-|y-x|,y) \big|
\int_{\R^3_v}\langle v\rangle^4 \big| \widehat{Z}^{p_2}f (t-|y-x|,y,v) \big| \,dv\frac{dy}{|y-x|}, \]
where we also used Lemma \ref{Lemkernel} to estimate $\mathbf{b}^\mu (\widehat{v},\omega)$ and its $v$-derivatives. Estimating $A_J$ through \eqref{equa:fortheproof} and using Corollary \ref{Cordecayintvbof}, we obtain
\[ \big[\nabla_{t,x}  ZA_I \big]_{\T_0} \lesssim \Lambda \varepsilon \mathbf{Y}^1_{4-2\eta}(t,x) \leq \Lambda \varepsilon \mathbf{Y}^1_{4-1/4}(t,x), \]
where $\mathbf{Y}^1_{4-1/4}(t,x)$ is defined and estimated in Lemma \ref{Lemint}.
\end{proof}

\subsection{Optimal decay estimates for momentum averages}

We would now like to improve the decay estimates provided by Corollary \ref{Corboundednessbof}, whose proof relies on the strategy used in the linearised setting. As we shall see, $f$ exhibits modified scattering behavior, and $\mathcal{X}(0,t,x,v)$ deviates logarithmically from the linear characteristic $x-t\widehat{v}$. In view of this observation and the discussion in Section \ref{Subsubsecmachin}, we seek a more accurate approximation of the nonlinear spatial characteristics of a suitable form.

The analysis developed in this section will also be useful for the study of modified scattering. Accordingly, whenever $T=+\infty$, we complement several of the statements below with their corresponding asymptotic conclusions. 

In contrast to $f$, which does not exhibit linear asymptotic behavior, its spatial average, being a function of $v$ only, converges as $t \to + \infty$. This quantity plays a crucial role in our analysis since, as in the linearised setting, it governs the asymptotic behavior of the source terms in the wave equations.

\begin{Pro}\label{Prospaave}
Let $\mathbf{Q}_t(v) \in C^1([0,T[ \times \R^3_v)$ be the weighted spatial average of $f$,
\[ \mathbf{Q}_t(v) \coloneqq \langle v \rangle^5 \int_{\R^3_x} f(t,x,v) \dr x. \]
Then, we have
\[ \forall \, (t,v) \in [0,T[ \times \R^3_v, \qquad \qquad \langle v \rangle^9 \big| \mathbf{Q}_t(v) \big| \lesssim \varepsilon. \]
Moreover, if $T=+\infty$, there exists $\mathbf{Q}_\infty \in C^0(\R^3_v,\R)$ such that,
\[ \forall \, (t,v) \in \R_+ \times \R^3_v, \qquad \qquad \langle v \rangle^9 \big|\mathbf{Q}_t(v) - \mathbf{Q}_\infty (v) \big| \lesssim \varepsilon \, \langle t \rangle^{-\frac{1}{2}} .\]
\end{Pro}
\begin{proof}
    We start by writing $\partial_t f+ \widehat{v} \cdot \nabla_x f= - \nabla \mathbf{A} \cdot \nabla_v f$. Then, we use the relation
    \begin{equation}\label{dvforscatt}
    \langle v \rangle\partial_{v^j} = \widehat{\Omega}_{0j} -x^j \partial_t -t\partial_{x^j} = \widehat{\Omega}_{0j} -(x^j-\widehat{v}^jt) \partial_t-\widehat{v}^jS+\widehat{v}^j(x-t\widehat{v}) \cdot \nabla_x+t \widehat{v}^j \widehat{v} \cdot \nabla_x -t\partial_{x^j}.
    \end{equation}
Integration by parts yields
\[ \bigg| \partial_t \int_{\R^3_x} f(t,x,v) \dr x \bigg| \lesssim \sum_{1 \leq I \leq K} \sum_{\widehat{Z} \in \Pp_S} \int_{\R^3_x} \langle x - t \widehat{v} \rangle \big| \nabla \mathbf{A} \big|(t,x) \big| \widehat{Z} f \big| (t,x,v) \dr x + \int_{\R^3_x} \langle t \rangle \big| \nabla_{t,x} \nabla \mathbf{A} \big|(t,x) \big|  f \big| (t,x,v) \dr x. \]
Recall now from Lemma \ref{Lemweight} that $\langle t+|x| \rangle \lesssim \langle t-|x| \rangle \, \langle x - t \widehat{v} \rangle \, \langle v \rangle^2$. Thus, estimating $\nabla \mathbf{A}$ using Lemma \ref{Lemsource}, we get
\[ \bigg| \partial_t \int_{\R^3_x} f(t,x,v) \dr x \bigg| \lesssim \frac{1}{\langle t \rangle^{\frac{7}{4}}} \sum_{\widehat{Z} \in \Pp_S}  \sup_{x \in \R^3_x} \langle x - t \widehat{v} \rangle^6 \langle v \rangle^4 \Big( |f|(t,x,v)+ \big| \widehat{Z} f \big| (t,x,v) \Big) . \]
Applying Corollary \ref{Corboundednessbof}, we obtain
\begin{equation}\label{eq:estiQt}
\big| \langle v \rangle^9  \partial_t \mathbf{Q}_t (v) \big| \lesssim \varepsilon \, \langle t \rangle^{- \frac{3}{2}}.
\end{equation}
We can conclude the proof using this estimate and the smallness assumption on $f$.
\end{proof}

It allows to deduce an expansion for the source term of the wave equations in \eqref{VW}.

\begin{Cor}\label{Corexp}
    Let $I \in \llbracket 1 , K \rrbracket$ and $\mathbf{a} \in C^1(\R^3_v)$ such that $|\mathbf{a}(v)|+\langle v \rangle |\nabla_v \mathbf{a}(v)| \lesssim \langle v \rangle^6$. Then, we have
    \[ \forall \, (t,x) \in [0,T[ \times \R^3, \qquad \qquad \qquad \bigg| \, t^3\int_{\R^3_v} \mathbf{a} (v) f(t,x,v) \dr v - \mathbf{a} \bigg( \frac{\widecheck{\; x \;}}{t} \bigg) \mathbf{Q}_t \bigg( \frac{\widecheck{\; x \;}}{t} \bigg) \mathds{1}_{|x|<t} \bigg| \lesssim \frac{\varepsilon}{\langle t \rangle^{\frac{1}{2}}}. \]
    If $T=+\infty$, the previous estimate remains valid with $\mathbf{Q}_t$ replaced by $\mathbf{Q}_\infty$.
\end{Cor}
\begin{proof}
 Assume first that $|x| < t$ and $t \geq 1$. Let $g(t,x,v) \coloneqq \mathbf{a}(v) f(t,x+t\widehat{v},v)$ and apply \cite[Lemma~4.12]{scat} to get
  \begin{align*} \bigg| \int_{\R^3_v}g(t,x-t\widehat{v},v) \dr v  - \int_{\R^3_x} \big[ \langle v \rangle^5 g] \bigg( t,y, \frac{\widecheck{ \; x \;}}{t} \bigg) \dr y \bigg| \lesssim \frac{1}{t} \sup_{(y,w) \in \R^3 \times \R^3} \langle y \rangle^5\langle w \rangle^7 \big( |g|(t,y,w)+\langle w \rangle \big| \nabla_ v g \big|(t,y,w) \big). 
  \end{align*}
Then, note that Lemma \ref{Lemrelftoh} and the assumption on $\mathbf{a}(v)$ imply
  \[ \langle y \rangle^5\langle w \rangle^7 \big( |g|(t,y,w)+\langle w \rangle \big| \nabla_ v g \big|(t,y,w) \big) \lesssim \sum_{\widehat{Z} \in \Pp_S} \langle y-t\widehat{w} \rangle^6\langle w \rangle^{13} \big( |f|(t,y-t\widehat{w},w)+ \big| \widehat{Z} f \big|(t,y - t \widehat{w},w) \big) \lesssim \varepsilon \, \langle t \rangle^{\frac{1}{4}},\]
  where, in the last step, we applied Corollary \ref{Corboundednessbof}. To conclude the analysis of this case, we only need to observe that the spatial averages of $g$ and $f$ coincide and, when $T=+\infty$, apply Proposition \ref{Prospaave}. If $t \leq 1$, we estimate separately both terms using Corollary \ref{Cordecayintvbof} and Proposition \ref{Prospaave}. Finally, for the case $|x| \geq t$, we appeal to Corollary \ref{Cordecayintvbof}.
\end{proof}

The next result, together with Propositions \ref{Prodata} and \ref{ProAIT0}, shows that $\nabla_{t,x} A_I$ and $[\nabla_{t,x} A_I]_T$ have the same asymptotic behavior along timelike curves. 

\begin{Pro}\label{Proeffective}
For any $0 \leq \mu \leq 3$, there exists $\mathbb{A}_I^\mu \in C^0(\R^3_v,\R)$ such that, for all $(t,v) \in [0,T[ \times \R^3_v$,
\[ \big| t^2 \big[ \partial_{x^\mu} A_I \big]_T(t,t\widehat{v})-\mathbb{A}_I^\mu (t,v) \big| \lesssim \varepsilon t^{-\frac{1}{2}} \langle v \rangle^7 , \]
where
\[ \mathbb{A}_I^\mu (t,v) \coloneqq \int_{\substack{|z| \leq 1 \\  |z+\widehat{v}| < 1-|z|  }} \mathbf{a}^\mu_I \! \left(\frac{z+\widehat{v}}{1-|z|} , \frac{z}{|z|}\right) \mathbf{Q}_t \left(\frac{\widecheck{z+\widehat{v}}}{1-|z|} \right) \frac{\dr z}{|z|^2(1-|z|)^3}. \]
If $T=+\infty$, the estimate remains valid with $\mathbb{A}_I^\mu (t,v)$ replaced by $\mathbb{A}_I^\mu (\infty,v)$.
\end{Pro}
\begin{Rq}\label{Rqdomainint}
On the domain of integration, we have $4(1-|z|) \geq \langle v \rangle^{-2}$. This follows from 
\[1-|\widehat{v}| = 1-|z|+|z|-|\widehat{v}| \leq 1-|z|+|z+\widehat{v}| \leq 2(1-|z|). \]
\end{Rq}
\begin{proof}
Let $(t,w) \in [0,T[ \times \R^3_v$. In what follows, $\mathbb{A}_I^\mu (v)$ denotes either $\mathbb{A}_I^\mu (t,v)$ or, when $T=+\infty$, $\mathbb{A}_I^\mu (\infty,v)$. Recall from Proposition \ref{ProGS} the expression of $\big[ \partial_{x^\mu} A_I \big]_T(t,t\widehat{w})$. Performing the change of variables $z(y)=(y-t \widehat{w})/t$ yields
\[t^2 \big[ \partial_{x^\mu} A_I \big]_T(t,t\widehat{w})= t^3\int_{|z|\leq 1}\int_{\R^3_v}
 \mathbf{a}^\mu_I \big( \widehat{v},\omega \big) f \big( t(1-|z|),t (\widehat{w}+z),v \big) \dr v \frac{\dr z}{|z|^2}, \qquad \qquad \omega \coloneqq \frac{z}{|z|}. \]
Note now that $|\mathbf{a}^\mu(\widehat{v},\omega)|+\langle v \rangle |\nabla_v \mathbf{a}(\widehat{v},\omega)| \lesssim \langle v \rangle^6$, uniformly in $\omega \in \mathbb{S}^2$, according to Lemma \ref{Lemkernel}. In the forthcoming computations, we will use that
\begin{equation}\label{eq:frakI}
\mathfrak{I}_a(w) \coloneqq \int_{\substack{|z| \leq 1 \\  |z+\widehat{w}| < 1-|z|  }} \frac{\dr z}{|z|^2(1-|z|)^a} \lesssim \langle w \rangle^{2a} , \qquad \qquad a \geq 0, 
\end{equation}
which follows from Remark \ref{Rqdomainint}. Thus, Corollary \ref{Corexp} yields to
\begin{equation}\label{eq:breman}
\Big|t^2 \big[ \partial_{x^\mu} A_I \big]_T(t,t\widehat{w}) - \mathbb{A}_I^\mu (w) \Big| \lesssim \frac{\varepsilon}{\langle t \rangle^{\frac{1}{2}}} \mathfrak{I}_{\frac{7}{2}}(x)+t^3\int_{\substack{|z| \leq 1 \\  |z+\widehat{v}| \geq 1-|z|  }} \int_{\R^3_v} \mathbf{a}^\mu_I \big( \widehat{v},\omega \big) f \big( t(1-|z|),t (\widehat{w}+z),v \big) \dr v \frac{\dr z}{|z|^2}. 
\end{equation}
In the last term on the right hand side, $f$ is always evaluated in the exterior of the light cone. As a consequence, we can use the improved decay estimates in Corollary \ref{Cordecayintvbof}. It yields
\[  \bigg| \int_{\substack{|z| \leq 1 \\  |z+\widehat{v}| \geq 1-|z|  }} \int_{\R^3_v} \mathbf{a}^\mu_I \big( \widehat{v},\omega \big) f \big( t(1-|z|),t (\widehat{w}+z),v \big) \dr v \frac{\dr z}{|z|^2} \bigg| \lesssim   \int_{|z| \leq 1} \frac{\varepsilon \dr z }{\langle t(1-|z|)+t |\widehat{w}+z| \rangle^{\frac{15}{4}}|z|^2}  . \] 
The reverse change of variables $y=t\widehat{w}+tz$ then allows, using Lemma \ref{Lemint}, to bound the last term on the right hand side of \eqref{eq:breman} by
\[\varepsilon t^2\mathbf{Y}^2_{4-1/4}(t,t\widehat{w}) \lesssim \varepsilon t^{\frac{1}{4}}\langle t(1-|\widehat{w}|) \rangle^{-1} \lesssim \varepsilon t^{-\frac{3}{4}} \langle w \rangle^2. \]
\end{proof}

We will require estimates for the first order derivatives of $\mathbb{A}_I^\mu(t,v)$.

\begin{Lem}\label{LemJacAbb}
    Let $0 \leq \mu \leq 3$ and $1 \leq I \leq K$. Then, for all $(t,v) \in [0,T[ \times \R^3_v$, we have
    \[  \big| \mathbb{A}_I^\mu \big|(t,v) \lesssim \varepsilon \, \langle v \rangle^6, \qquad \qquad  \big| \nabla_v \mathbb{A}_I^\mu \big|(t,v) \lesssim \varepsilon \,  \langle v \rangle^7 \langle t \rangle^{\frac{1}{4}}, \qquad \qquad \big| \partial_t \mathbb{A}_I^\mu \big|(t,v) \lesssim \varepsilon \, \langle v \rangle^6 \langle t \rangle^{-\frac{3}{2}}.  \]
\end{Lem}
\begin{proof}
Recall from Lemma \ref{Lemcdv} that $|\nabla_u \widecheck{u}| \lesssim \langle \widecheck{u} \rangle^3$. Since $| \mathbf{a}^\mu_I |(\widehat{v},\omega)+|\nabla_v \mathbf{a}^\mu_I |(\widehat{v},\omega) \lesssim \langle v \rangle^6$ according to Lemma \ref{Lemkernel}, we have
\[ \big|\mathbb{A}_I^\mu \big|(t,v) \lesssim \mathfrak{I}_3(v) \sup_{w \in \R^3_v} \langle w \rangle^6 | \mathbf{Q}_t (w) |, \qquad \qquad \big|\nabla_v \mathbb{A}_I^\mu \big|(t,v) \lesssim  \frac{1}{\langle v \rangle}\mathfrak{I}_4(v) \sup_{w \in \R^3_v} \langle w \rangle^9 \big( | \mathbf{Q}_t (w) |+ | \nabla_v \mathbf{Q}_t (w) | \big), \]
where $\mathfrak{I}_3(v)$ and $\mathfrak{I}_4(v)$ are defined and bounded by $\langle v \rangle^8$ in \eqref{eq:frakI}. By Proposition \ref{Prospaave}, $\langle w \rangle^9  | \mathbf{Q}_t (w) | \lesssim \varepsilon$. For the last term, we observe that
\[ \langle v \rangle \, \partial_{v^i}\int_{\R^3_x} f \dr x = \int_{\R^3_x} \widehat{\Omega}_{0i} f - t \partial_{x^i} f-x^i \partial_t f \dr x =  \int_{\R^3_x} \widehat{\Omega}_{0i} f -(x^i-t\widehat{v}^i) \partial_t f - \widehat{v}^i Sf-3\widehat{v}^i f \dr x, \]
where, in the second step, we performed integration by parts. It yields, using Corollary \ref{Corboundednessbof},
\[ \langle v \rangle^9 | \nabla_v \mathbf{Q}_t (v) | \lesssim \sum_{\widehat{Z} \in \mathbb{P}_S} \sup_{x \in \R^3_x} \langle x -t \widehat{v} \rangle^5 \langle v \rangle^{8} \big( \big| \widehat{Z} f \big|(t,x,v)+|f|(t,x,v) \big) \lesssim \varepsilon \langle t \rangle^\eta \log^6 \, \langle t+1 \rangle. \]
Finally, we have
\[  \big| \partial_t \mathbb{A}_I^\mu \big|(t,v) \lesssim \mathfrak{I}_3(v) \sup_{w \in \R^3_v} \langle w \rangle^6 \big| \partial_t \mathbf{Q}_t \big|(w), \]
which, together with \eqref{eq:estiQt} and \eqref{eq:frakI}, implies the result.
\end{proof}

This yields the following.

\begin{Cor}\label{Coreffective}
 Let $\mathscr{D} \in C^0 \cap L^\infty (\R^3_v,\R^3_v)$ and $I \in \llbracket 1 , K \rrbracket$. Then, for any $0 \leq \mu \leq 3$, we have
 \[ \forall \, (t,x,v) \in [0,T[ \times \R^3_x \times \R^3_v, \qquad \qquad \big| t^2\partial_{x^\mu} A_I \big( t,x+t\widehat{v}+\mathscr{D}(v) \log (1+t)  \big)- \mathbb{A}_I^\mu (t,v) \big| \lesssim \Lambda \frac{\langle x \rangle^2 \langle v \rangle^7}{\langle t \rangle^{\frac{1}{2}}} . \]
 If $T=+\infty$, the estimate remains valid with $\mathbb{A}_I^\mu (t,v)$ replaced by $\mathbb{A}_I^\mu (\infty,v)$.
\end{Cor}
\begin{proof}
Let us first show that
\begin{align*}
t^2 \big| \partial_{x^\mu} A_I \big( t,x+t\widehat{v}+\mathscr{D}(v) \log (1+t)  \big)-  \partial_{x^\mu} A_I \big( t,t\widehat{v} \big) \big| & \lesssim \sup_{0 \leq \theta \leq 1} t^2 \big| \nabla^2_{t,x} A_I \big| \big( t,\theta (x + \mathscr{D}(v)\log (1+t)) +t \widehat{v} \big) \\
& \lesssim \Lambda \frac{\langle x \rangle^2 \langle v \rangle^4\log^2 (1+t)}{\langle t \rangle^{1-\eta} } . 
\end{align*}
The first inequality follows from the mean value theorem. For the second one, we use first the decay estimates given by \eqref{equa:fortheproof}--\eqref{eq:extragain}, and the inequality $\langle t+|y | \rangle \lesssim \langle t-|y| \rangle \langle y-t \widehat{v} \rangle \langle v \rangle^2 $, applied with $y = \theta (x+\mathscr{D}(v) \log (t) )+t\widehat{v}$, which is given by Lemma \ref{Lemweight}. Next, we use the Glassey-Strauss decomposition of $\partial_{x^\mu} A_I$ provided by Proposition \ref{ProGS}. Finally, Propositions \ref{Prodata} and \ref{ProAIT0} yield
\[ \Big|  \big[\partial_{x^\mu} A_I \big]_{\mathrm{data}} (t,t\widehat{v} \big) \Big|+  \Big|  \big[\partial_{x^\mu} A_I \big]_{\T_0} (t, t\widehat{v} \big) \Big| \lesssim \Lambda  \langle v \rangle^{\frac{7}{2}} \langle t \rangle^{-\frac{11}{4}}  , \]
whereas $[\partial_{x^\mu} A_I ]_{T}(t,t\widehat{v})$ is handled using Proposition \ref{Proeffective}.
\end{proof}

To derive pointwise decay estimates allowing to close our argument, we will first show that $\mathcal{X}(0,t,x,v)$ can be well-approximated by a quantity of the form $x-t\widehat{v}+\mathscr{C}_t(v) \log(1+t)$. Motivated by
\begin{equation}\label{eq:TAxtv}
\T_{\mathbf{A}} \big (x-t\widehat{v} \big) = -\frac{t}{\langle v \rangle} \nabla \mathbf{A} +\frac{t \widehat{v}}{\langle v \rangle} \nabla \mathbf{A} \cdot \widehat{v},
\end{equation}
the definition \eqref{defA} of $\nabla \mathbf{A}$ and Proposition \ref{Proeffective}, we introduce the following correction to the linear spatial characteristics. For $t \in [0,T[$, and also for $t=\infty$ when $T=+\infty$, we define
\begin{equation}\label{defCor} \mathscr{C}^i_t(v) \coloneqq -\sum_{1 \leq j \leq 3} \frac{\delta_{i,j}-\widehat{v}^i\widehat{v}^j}{\langle v \rangle} \sum_{1 \leq I \leq K} \sum_{0 \leq \mu \leq 3} \langle v \rangle^\delta Q^{i,\mu}_I(\widehat{v})\mathbb{A}_I^\mu (t,v), \qquad 1 \leq i \leq 3 .  
\end{equation}
\begin{Lem}\label{LemCtv}
    We have, for all $(t,x,v) \in [0,T[ \times \R^3_x \times \R^3_v$,
    \[ \big| x-t\widehat{v}-\mathscr{C}_t(v) \log(1+t)- \mathcal{X}(0,t,x,v) \big| \lesssim \Lambda \langle \mathcal{X} (0,t,x,v) \rangle^2 \langle \mathcal{V}(0,t,x,v) \rangle^{17}.  \]
    Moreover, we have $| \mathscr{C}_t |(v) \lesssim \varepsilon \, \langle v \rangle^6 $ and $|\nabla_v \mathscr{C}_t |(v) \lesssim \varepsilon \, \langle v \rangle^7 \langle t \rangle^{\frac{1}{4}}$ for all $(t,v) \in [0,T[ \times \R^3_v$.
\end{Lem}
\begin{proof}
The estimates for $ \mathscr{C}_t $ and $\nabla_v \mathscr{C}_t$ are a consequence of \eqref{defCor} and Lemma \ref{LemJacAbb}. Recall that $(t,x,v) \mapsto \mathcal{X}(0,t,x,v)$ is constant along the flow of $\T_{\mathbf{A}}$, and note that it coincides with $x-t\widehat{v}+\mathscr{C}_t(v) \log(1+t)$ at $t=0$. Next, we observe that \eqref{eq:TAxtv}--\eqref{defCor} imply
    \begin{align*}
        \big| \T_{\mathbf{A}} \big( x-t\widehat{v}-\mathscr{C}_t(v) \log(1+t) \big) \big| \lesssim  \!\sum_{1 \leq I \leq K} \sum_{0 \leq \mu \leq 3} & \, \frac{1}{1+t} \big| t^2 \partial_{x^\mu} A_I (t,x)- \mathbb{A}_I^\mu (t,v) \big|+   \big| \partial_{x^\mu} A_I (t,x)\big| \\
        & + \big| \partial_t \mathbb{A}_I^\mu (t,v) \big| \log \, \langle t \rangle + \big( | \mathbb{A}_I^\mu (t,v) |+|\nabla_v \mathbb{A}_I^\mu (t,v) | \big)  \big| \partial_{x^\mu} A_I (t,x)\big|  \log \, \langle t \rangle.
    \end{align*}
    Then, we apply Corollary \ref{Coreffective}, with $\mathscr{D}(v)=0$, and then Propositions \ref{ProVcharac}--\ref{ProVcharacbis} and Proposition \ref{Prozmod} to get
    \begin{equation*}
        \big| t^2 \partial_{x^\mu} A_I (t,x)- \mathbb{A}_I^\mu (t,v) \big| \lesssim \Lambda\frac{\langle x-t\widehat{v} \rangle^2 \langle v \rangle^7}{\langle t \rangle^{\frac{1}{2}}} \lesssim  \Lambda\frac{\langle \mathcal{X} (0,t,x,v) \rangle^2 \langle \mathcal{V}(0,t,x,v) \rangle^{17}}{\langle t \rangle^{\frac{1}{2}}} \log^2 \, \langle t+1 \rangle .
    \end{equation*}
    For the second term, we use first \eqref{equa:fortheproof} and then Lemma \ref{Lemweight} to obtain
    \[  \big| \partial_{x^\mu} A_I (t,x)\big| \lesssim \Lambda \langle t \rangle^{-1} \langle t-|x| \rangle^{-1} \lesssim \Lambda\langle t \rangle^{-2} \langle x-t\widehat{v} \rangle \, \langle v \rangle^2 \lesssim \Lambda \langle t \rangle^{-\frac{7}{4}}\langle \mathcal{X} (0,t,x,v) \rangle \, \langle \mathcal{V}(0,t,x,v) \rangle^6 . \]
    It allows to bound the fourth term by additionally using Lemma \ref{LemJacAbb} and $\langle v \rangle^7 \lesssim \langle \mathcal{V}(0,t,x,v) \rangle^{11}$. Finally, the third term can be controlled by using Lemma \ref{LemJacAbb} and $\langle v \rangle^6 \lesssim \langle \mathcal{V}(0,t,x,v) \rangle^9$. It yields
     \begin{align*}
        \big| \T_{\mathbf{A}} \big( x-t\widehat{v}-\mathscr{C}_t(v) \log(1+t) \big) \big| & \lesssim \Lambda \langle \mathcal{X} (0,t,x,v) \rangle^2 \langle \mathcal{V}(0,t,x,v) \rangle^{17} \langle t \rangle^{-\frac{5}{4}} .
    \end{align*}
    We conclude the proof using Duhamel's principle \eqref{eq:Duhamel} and Lemma \ref{LemforDuhamel}.
\end{proof}

The following result will be crucial in deriving optimal decay estimates for momentum averages.

\begin{Cor}\label{Cordecayintv}
 If $\varepsilon$ is small enough, there exists $C[\Lambda, \kappa]>0$ such that, for all $(t,x) \in [0,T[ \times \R^3_x$,
   \[ \mathcal{I}_{t,x} \coloneqq \int_{\R^3_v} \frac{\dr v}{ \langle x-t\widehat{v}+\mathscr{C}_t(v) \log (1+t)  \rangle^4  \, \langle v \rangle^{49}} \leq \frac{C}{\langle t+|x| \rangle^3}. \]
\end{Cor}
\begin{proof}
The goal is to adapt the strategy of the linear case. We consider several subdomains of $\R^3_v$.
\begin{itemize}
    \item If $|x| \geq 4t$ and $|x| \geq \langle v \rangle^6 \, \langle t \rangle^{\frac{1}{4}}$, we have $|x-t\widehat{v}+\mathscr{C}_t(v) \log (1+t)| \geq |x|/2$ by Lemma \ref{LemCtv} and the result follows.
    \item If $|x| \geq 4t$ and $|x| \leq \langle v \rangle^6 \, \langle t \rangle^{\frac{1}{4}}$, we obtain spatial decay by using the weight $\langle v \rangle^{-24}$.
      \item If $t \geq |x|/4$ and $t^{\frac{2}{3}} \geq \langle v \rangle^{10}$, we obtain time decay by using the weight $\langle v \rangle^{-45}$.
\item Finally, assume that $t \geq |x|/4$ and consider the domain
\[  \overline{\mathcal{D}} \coloneqq \big\{ v \in \R^3_v \; \big| \; \langle v \rangle^{10} \leq t^{\frac{2}{3}} \big\} . \]
Since the case $t \leq 1$ is straightforward, we assume that $t \geq 1$ and we perform the change of variables $y(v)=x-t\widehat{v}$. Using Lemma \ref{Lemcdv}, it yields
\[ t^3 \, \mathcal{I}_{t,x} =  \int_{|y-x|<t, \; \langle v(y) \rangle^{10} \leq t^{\frac{2}{3}}} \frac{\dr y}{\langle y+ \mathscr{C}_t ( v(y) )\log (1+t) \rangle^4 \langle v \rangle^5}, \qquad \qquad v(y) \coloneqq \frac{\widecheck{x-y}}{t} . \]
To conclude the proof, let us show that the map $\psi(y)=y+\mathscr{C}_t ( v(y) ) \log (1+t)$ has a Jacobian determinant bounded below by $1/2$. For this, note that the previous Lemma \ref{LemCtv} implies that, on the domain of integration
\[ \big| \nabla_y \big[ \mathscr{C}_t ( v(y) )  \big]\big| \lesssim t^{-1}  \big| \langle v (y) \rangle^3\nabla_v \mathscr{C}_t ( v(y) ) \big| \lesssim \varepsilon \, \langle v(y) \rangle^{10} t^{-\frac{3}{4}} \leq \varepsilon t^{-\frac{1}{9}}. \]
\end{itemize}
\end{proof}

We can now improve, for $t \geq |x|$, the estimates of Corollary \ref{Cordecayintvbof}.

\begin{Pro}\label{Prodecayvel}
For any $\widehat{Z} \in \Pp_S$ and all $(t,x,v) \in [0,T[ \times \R^3_x \times \R^3_v$, we have
\[ \int_{\R^3_v} \langle v \rangle^5 |f| (t,x,v) \dr v \lesssim \frac{\varepsilon}{\langle t+|x| \rangle^3} , \qquad \qquad \qquad  \int_{\R^3_v} \langle v \rangle^5 \big| \widehat{Z}f \big| (t,x,v) \dr v \lesssim \frac{\varepsilon}{\langle t+|x| \rangle^{3-\eta}}.\]
\end{Pro}
\begin{proof}
Note that for a distribution function $h$, we have
 \[ \int_{\R^3_v} \! \langle v \rangle^5 | h|(t,x,v) \dr v \leq \mathcal{I}_{t,x}  \sup_{w \in \R^3_v} \langle x-t\widehat{w}+ \mathscr{C}_t(w) \log (1+t) \rangle^{4}  \langle w \rangle^{54}  |h|(t,x,w). \]
The first factor $\mathcal{I}_{t,x}$ is bounded by Corollary \ref{Cordecayintv}. Using Propositions \ref{ProVcharac}--\ref{ProVcharacbis} and Lemma \ref{LemCtv}, we obtain the bounds 
 \[\langle w \rangle^{54} \lesssim \langle \mathcal{V}(0,t,x,w) \rangle^{81}, \qquad \qquad  \langle x-t\widehat{w}+ \mathscr{C}_t(w) \log (1+t) \rangle \lesssim \langle \mathcal{X}(0,t,x,w) \rangle^2 \,  \langle \mathcal{V}(0,t,x,w) \rangle^{17}. \]
We then control the second factor, with $h =  f$ and $h = \widehat{Z}f$, by applying Corollary \ref{Corboundedness}.
\end{proof}

We are now able to control sufficiently well the tangential term of the scalar fields to improve \eqref{eq:BA1}--\eqref{eq:BA2}.

\begin{Pro}\label{ProAIT}
    Let $I \in \llbracket 1 , K \rrbracket$ and $Z \in \mathbb{K}$. Then, for all $(t,x) \in [0,T[ \times \R^3$, we have
    \[ \big| \big[ \nabla_{t,x} A_I \big]_T \big|(t,x)+\langle t+|x| \rangle^{-\eta}\big| \big[ \nabla_{t,x} Z A_I \big]_T \big|(t,x) \lesssim \varepsilon \, \langle t+|x| \rangle^{-\frac{7}{4}} \langle t-|x| \rangle^{-\frac{1}{4}}. \]
    If $|x| \geq t$, then $\big| \big[ \nabla_{t,x} A_I \big]_T \big|(t,x) \lesssim \varepsilon \, \langle t+|x| \rangle^{-\frac{7}{4}} \langle t-|x| \rangle^{-1}$.
\end{Pro}
\begin{proof}
 To treat $A_I$ and $Z A_I$ simultaneously, let $ p \in \{0,1\}$ and estimate $Z^p A_I$. According to Proposition \ref{ProGS}, the bound $|\mathbf{a}^\mu (\widehat{v},\omega) | \lesssim \langle v \rangle^5$ given by Lemma \ref{Lemkernel} and, if $p=1$, the commutation formula of Proposition \ref{ProCom}, we have
 \[  \big| \big[ \nabla_{t,x} Z^p A_I \big]_T  \big|(t,x) \lesssim \int_{|y-x| \leq t } \int_{\R^3_v} \langle v \rangle^5 \big| f \big| \big( t-|y-x|,y,v \big) \dr v+\int_{\R^3_v} \langle v \rangle^5 \big| \widehat{Z}^p f \big| \big( t-|y-x|,y,v \big) \dr v \frac{\dr y}{|y-x|^2}. \]
 Using the decay estimates of the previous Proposition \ref{Prodecayvel}, we obtain
 \[  \big| \big[ \nabla_{t,x}  A_I \big]_T  \big|(t,x) \lesssim \varepsilon \mathbf{Y}^2_3(t,x), \qquad \qquad   \big| \big[ \nabla_{t,x} Z A_I \big]_T  \big|(t,x) \lesssim \varepsilon \, \langle t +|x| \rangle^{\eta} \mathbf{Y}^2_3 (t,x) , \]
 where $\mathbf{Y}^2_3(t,x)$ is defined and estimated in Lemma \ref{Lemint}. To derive the improved decay in the exterior of the light cone, we note that $|y| \geq |x|-|y-x| \geq t-|y-x|$ if $|x| \geq t$. Then, using the decay estimates of Corollary \ref{Corboundednessbof}, we obtain
 \[ \big| \big[ \nabla_{t,x} A_I \big]_T \big|(t,x) \lesssim \varepsilon \mathbf{Y}^2_{15/4} \lesssim \varepsilon \, \langle t+|x| \rangle^{-\frac{7}{4}} \langle t-|x| \rangle^{-1} .\]
\end{proof}

Combining Propositions \ref{ProAIT0} and \ref{ProAIT}, and recalling from Remark \ref{Rqconst} that their implicit constants are independent of $C_{\mathrm{boot}}$, we improve the bootstrap assumptions \eqref{eq:BA1}--\eqref{eq:BA2} by choosing $C_{\mathrm{boot}}$ sufficiently large.

\section{Modified scattering dynamics}\label{SEcscatt}

\subsection{Modified scattering for the distribution function}

The strategy for proving such a result is now standard (see, for instance, \cite[Section~1.3]{scatmap}). Moreover, most of the required analysis has already been carried out in the proof of global existence. Indeed, rewriting $\nabla_v f$ using \eqref{dvforscatt}, we have
\begin{align}
 \nonumber   \partial_t \big[ f(t,x+t\widehat{v},v) \big] & = -\big[ \nabla \mathbf{A} \cdot \nabla_v f \big](t,x+t\widehat{v},v) \\
    & = \frac{t}{\langle v \rangle} \big[ \nabla \mathbf{A} \cdot \big(  \nabla_x  f - \widehat{v}  \big(\widehat{v} \cdot \nabla_x f\big) \big) \big](t,x+t\widehat{v},v)-\frac{1}{\langle v \rangle} \big[ \nabla \mathbf{A} \cdot \widehat{\mathbf{Z}}_{\mathrm{good}} f \big](t,x+t\widehat{v},v), \label{eq:nolinscat}
\end{align}
where the quantity $\widehat{\mathbf{Z}}_{\mathrm{good}} f (t,y,v) \in \R^3$ is defined as
$$ \big[\widehat{\mathbf{Z}}_{\mathrm{good}} f \big]^j(t,y,v) \coloneqq  \big[ \widehat{\Omega}_{0j}f -(y^j-\widehat{v}^jt) \partial_tf-\widehat{v}^jSf+\widehat{v}^j(x-t\widehat{v}) \cdot \nabla_xf \big] (t,y,v).$$
Estimating $\nabla \mathbf{A}$ through Lemma \ref{Lemsource} and using Lemma \ref{Lemweight}, we get
\begin{equation}\label{eq:part1}
  \frac{1}{\langle v \rangle}  \Big|  \nabla \mathbf{A} \cdot \widehat{\mathbf{Z}}_{\mathrm{good}} f   \Big| (t,y,v) \lesssim \frac{\Lambda}{\langle t \rangle^{2}} \sup_{\widehat{Z} \in \Pp_S} \langle y - t \widehat{v} \rangle^2 \langle v \rangle^2 \big| \widehat{Z} f \big|(t,y,v) \lesssim \Lambda \varepsilon \frac{ \log^6(t)}{\langle t \rangle^{2-\eta}},
\end{equation}
where, in the last step, we applied Corollary \ref{Corboundednessbof}. However, according to Corollary \ref{Coreffective}, the first term on the right hand side of \eqref{eq:nolinscat} is not expected to be integrable in time (see \cite{Emile2} for more details in the context of the Vlasov-Maxwell system). To obtain convergence of the distribution function, we then need to consider a better approximation of the nonlinear spatial characteristics, namely $x+t\widehat{v}+\mathscr{C}_\infty(v) \log (t)$, where $\mathscr{C}_\infty (v)$ is defined in \eqref{defCor}. Then, from \eqref{eq:nolinscat} and \eqref{eq:part1}, we get
\begin{align}
 \nonumber \Big|  \partial_t \big[ f(t,x+t\widehat{v}+\mathscr{C}_\infty(v) \log(t),v) \big] \Big| & \lesssim \bigg| \frac{t}{\langle v \rangle} \nabla \mathbf{A}(t,x+t\widehat{v}+\mathscr{C}(v) \log(t),v)+\frac{\mathscr{C}(v)}{t}\bigg| \cdot \big|\nabla_x f(t,x+t\widehat{v},v) \big|+\frac{\varepsilon \log^6(t)}{\langle t \rangle^{2-\eta}}.
\end{align}
In view of the definition \eqref{defA} of $\nabla \mathbf{A}$, we can estimate the first factor using Corollary \ref{Coreffective}, with $\mathscr{D}=\mathscr{C}_\infty$. Then, using again Corollary \ref{Corboundednessbof}, we finally obtain
\[  \big|  \partial_t \big[ f(t,x+t\widehat{v}+\mathscr{C}_\infty(v) \log(t),v) \big] \big|  \lesssim \varepsilon \, \langle t \rangle^{-\frac{3}{2}} \log^6(t), \]
which allows to derive modified scattering for $f$.

\subsection{Linear scattering for the scalar fields}

Let $1 \leq I \leq K$. Denoting by $\slashed{\Delta}$ the Laplacian on $\mathbb{S}^2$, we write
\[ \Box A_I = -\frac{1}{r} \big(\partial_t - \partial_r \big)\big( \partial_t + \partial_r \big)(rA_I)+ \frac{1}{r^2} \slashed{\Delta} A_I. \]
Since we do not propagate norms strong enough of the solutions, the last term on the right hand side is not controlled sufficiently well to show strong convergence for $r A_I$. It would require to commute twice by $\Omega_{ij}$. This is why we work in a weak topology. 

We then consider a test function $\Phi \in C_c^\infty (\mathbb{S}^2)$ and let us denote by $\langle \cdot , \cdot \rangle_{L^2(\mathbb{S}^2)}$ the scalar product of $L^2(\mathbb{S}^2)$. We will use the coordinates $(\underline{u},u)=(t+|x|,t-|x|)$. Performing integration by parts, we get
\begin{align}\label{eq:machindelafin}
 4 \Big| \partial_u \partial_{\underline{u}} \big\langle r A_I , \Phi \big\rangle_{L^2(\mathbb{S}^2)}  \Big| (u,\underline{u}) & \leq \frac{1}{r} \Big|  \big\langle  A_I , \slashed{\Delta}\Phi \big\rangle_{L^2(\mathbb{S}^2)} \Big| (u,\underline{u})  + \Big|  \big\langle \pmb{\rho}_I [f] , \Phi \big\rangle_{L^2(\mathbb{S}^2)} \Big| (u,\underline{u}) .
 \end{align}
The next step consists in proving an estimate for $A_I$.
\begin{Lem}\label{LemAIesti}
We have, for all $(t,x) \in \R_+ \times \R^3$,
\[ \big| A_I \big| (t,x) \lesssim \left\{
    \begin{array}{ll}
        \Lambda \langle t+|x| \rangle^{-1} & \mbox{if } t >|x|, \\
        \Lambda \langle t+|x| \rangle^{-1} \langle t -|x| \rangle^{-\frac{3}{4}} & \mbox{if }  t \leq |x| .
    \end{array}
\right. \]
\end{Lem}
\begin{proof}
    Let $(t,r\omega) \in \R_+ \times \R^3$, with $\omega \in \mathbb{S}^2$. To derive the stated estimate, we will integrate in $u$, for fixed $\underline{u}=t+r$ and $\omega$. We note that $2\partial_{\underline{u}}=\partial_t + \partial_r$ and $2 \partial_u = \partial_t - \partial_r$. If $r \geq t$, we have
    \[ A_I(t,r \omega) = A_I \big(0,(t+r)\omega \big) + \int_{u=-t-r}^{t-r} \frac{1}{2} \big[ \partial_t A_I-\partial_r A_I \big] \Big( \frac{t+r+u}{2}, \frac{t+r-u}{2} \omega \Big) \dr u. \]
    Let us justify that
    \[ |A_I|(t,r\omega) \lesssim \frac{\Lambda}{\langle t+r \rangle^2}+ \int_{u=-t-r}^{t-r} \frac{\Lambda}{\langle t+r \rangle \, \langle u \rangle^{\frac{7}{4}}} \dr u \lesssim \frac{\Lambda}{\langle t + r \rangle \, \langle t-r \rangle^{\frac{3}{4}}} . \]
   For this, we use first the Glassey-Strauss decomposition of $\nabla_{t,x} A_I$ given by Proposition \ref{ProGS}. Then we use Propositions \ref{Prodata} and \ref{ProAIT0} to control $[\nabla_{t,x} A_I]_{\mathrm{data}}$ and $[\nabla_{t,x} A_I]_{\T_0}$, respectively, and bound $[\nabla_{t,x} A_I]_{T}$ using the decay estimates of Proposition \ref{ProAIT} in the exterior of the light cone.
   
    Assume now that $t >r$. We integrate again in $u$, from the point $\big( \frac{t+r}{2},\frac{t-r}{2}\omega \big)$ to $(t,r\omega)$. Since $\nabla_{t,x} A_I$ decays at a weaker rate in the interior of the light cone according to Propositions \ref{Prodata}, \ref{ProAIT0} and \ref{ProAIT}, it yields
    \[ |A_I|(t,r\omega) \lesssim \frac{\Lambda}{\langle t+r \rangle}+ \int_{u=0}^{t-r} \frac{\Lambda}{\langle t+r \rangle \langle u \rangle^{\frac{7}{4}}}+\frac{\varepsilon}{\langle t+r \rangle^{\frac{7}{4}}  \langle u \rangle^{\frac{1}{4}}} \dr u \lesssim \frac{\Lambda}{\langle t + r \rangle } . \]
\end{proof}

To estimate the left hand side in \eqref{eq:machindelafin}, we use the previous Lemma \ref{LemAIesti} as well as Proposition \ref{Prodecayvel} to control $\pmb{\rho}_I[f]$. Moreover, to deal with the region where $r \leq 1+t/2$, we observe that
\[ \frac{1}{r} \Big|  \big\langle  A_I , \slashed{\Delta}\Phi \big\rangle_{L^2(\mathbb{S}^2)} \Big| (u,\underline{u})  = \frac{1}{r} \Big|  \big\langle \slashed{\nabla} A_I , \slashed{\nabla}\Phi \big\rangle_{L^2(\mathbb{S}^2)} \Big| (u,\underline{u}) \lesssim  \big\| \nabla_{t,x} A_I  \big\|_{L^2(\mathbb{S}^2)}(u,\underline{u}) \lesssim \frac{\Lambda}{ \langle \underline{u} \rangle \, \langle u \rangle} \lesssim \frac{\Lambda}{\langle \underline{u} \rangle^2}, \]
where, in the last step, we used \eqref{equa:fortheproof}. Since $r \gtrsim 1+t+r$ if $r \geq 1+t/2$, it yields
\begin{align*}
 \Big| \partial_u \partial_{\underline{u}} \big\langle r A_I , \Phi \big\rangle_{L^2(\mathbb{S}^2)}  \Big| (u,\underline{u}) \lesssim  \left\{
    \begin{array}{ll}
        \Lambda \langle \underline{u} \rangle^{-2} & \mbox{if } u>0, \\
        \Lambda  \langle \underline{u} \rangle^{-2} \langle u \rangle^{-\frac{3}{4}} & \mbox{if } u \leq 0 .
    \end{array}
\right. 
 \end{align*}
The following computations are similar to the ones carried out in the proof of Lemma \ref{LemAIesti}. For all $\underline{u} \geq 0$ and $- \underline{u} \leq u \leq 0$, we have
\[ \Big| \partial_{\underline{u}} \big\langle r A_I , \Phi \big\rangle_{L^2(\mathbb{S}^2)}  \Big| (u,\underline{u}) \lesssim  \Lambda \langle \underline{u} \rangle^{-\frac{7}{4}}.\]
We then deduce, that for $\underline{u} \geq  u \geq 0$,
 \[ \Big| \partial_{\underline{u}} \big\langle r A_I , \Phi \big\rangle_{L^2(\mathbb{S}^2)}  \Big| (u,\underline{u}) \lesssim  \frac{\Lambda}{ \langle \underline{u} \rangle^{\frac{7}{4}}}+ \int_{u'=0}^{u} \frac{\Lambda}{\langle \underline{u} \rangle^2} \lesssim \Lambda \frac{\langle u \rangle^{\frac{3}{4}}}{ \langle \underline{u} \rangle^{\frac{7}{4}}}.\]
As a consequence, there exists $A_I^\infty \in C^0 \in L^\infty (\R , L^2 (\mathbb{S}^2))$ such that 
\[ \Big| \big\langle r A_I (u, \underline{u},\omega) - A_I^\infty(u,\omega) , \Phi (\omega)\big\rangle_{L^2(\mathbb{S}^2_\omega)}  \Big| \lesssim \Lambda \frac{\langle u \rangle^{\frac{3}{4}}}{ \langle \underline{u} \rangle^{\frac{3}{4}}}.\]

\section{Instabilities}

The purpose of this section is to prove Propositions \ref{ProInstaLambdaLarge} and \ref{ProInstaInidecay}. We set $K=1$ and we denote the scalar field $A_1$ simply by $A$. We assume further that $\nabla \mathbf{A} \cdot \nabla_v f= \langle v \rangle^\delta \partial_{x^1} A \partial_{v^1}f$, so that the characteristics associated with $\T_{\mathbf{A}}$ satisfy
\begin{equation}\label{characInsta}
    \dot{\mathcal{X}} = \widehat{\mathcal{V}}, \qquad \qquad \dot{\mathcal{V}}^1 = \langle \mathcal{V} \rangle^\delta \partial_{x^1} A \big( \cdot , \mathcal{X} \big), \quad \dot{\mathcal{V}}^2=0, \quad \dot{\mathcal{V}}^3 = 0.
\end{equation} 

\subsection{Preliminary estimates}

The homogeneous part of the scalar field will be given by $\Lambda \mathcal{A}_{\mathrm{hom}}$, introduced in Definition \ref{Definsta} and that we bound in certain regions in the next result.

\begin{Lem}\label{LemInsta0}
There exists $c >0$ such that, for all $(t,x) \in \R_+ \times \R^3$ satisfying $1 \leq |x| -t \leq 10$ and $x^1/|x| \geq 1/4$, we have
\[  \partial_{x^1} \mathcal{A}_{\mathrm{hom}}(t,x) \geq 4c (1+t)^{-1} . \]
\end{Lem}
\begin{proof}
Recall that $\mathcal{A}_{\mathrm{hom}}$ is spherically symmetric and satisfies $|x|\mathcal{A}_{\mathrm{hom}}= \mathcal{A}(-|x|)$. Moreover $\mathcal{A}$ is supported in $\R_-$. Thus, using the notations $r=|x|$ and $\omega^i = x^i/|x|$, we have
\[  \mathcal{A}_{\mathrm{hom}} (t,x) = \frac{ \mathcal{A}(t-r)}{r},  \qquad \qquad   \partial_{x^1} \mathcal{A}_{\mathrm{hom}}(t,x) = -\omega^1 \frac{\mathcal{A}'(t-r)}{r}-\omega^1 \frac{\mathcal{A}(t-r)}{r^2} .  \]
Recall that $\mathcal{A}(u)=e^{-2 u}$ for all $-10 \leq u \leq -1$. Hence, for all $(t,x) \in \R_+ \times \R^3$ such that  $-10 \leq t-r \leq -1$ and $\omega^1 \geq 1/4$, we have
\begin{equation*}
 \partial_{x^1} \mathcal{A}_{\mathrm{hom}}(t,x) \geq \frac{e^{2 (t-r)}}{4r} \geq \frac{e^{-10 }}{4(1+t)}.
 \end{equation*}
\end{proof}

In all the settings considered below, the scalar field $A$ will satisfy the following estimate. Let $\varepsilon , \, \Lambda , \, C_{\mathrm{stab}} >0$ and assume that for all $1 \leq |x| - t \leq 10$,
\begin{align}
  \langle t \rangle  \big| \partial_{x^1} A - \Lambda \partial_{x^1} \mathcal{A}_{\mathrm{hom}}\big|(t,x)  & \leq C_{\mathrm{stab}} \varepsilon . \label{eq:decayinsta} 
\end{align}
We now estimate the scalar field under this assumption.
\begin{Lem}\label{LemdecayAinsta}
   If $C_{\mathrm{stab}} \varepsilon / \Lambda$ is small enough, we have, for all $(t,x) \in \R_+ \times \R^3$ such that $1 \leq |x|-t \leq 10$ and $x^1/|x| \geq 1/4$,
    \[   \partial_{x^1} A(t,x) \geq 2c\Lambda (1+t)^{-1}.  \]
\end{Lem}
\begin{proof}
    This follows from Lemma \ref{LemInsta0} and \eqref{eq:decayinsta}.
\end{proof}

Next, we establish several properties of certain characteristics associated with the operator $\T_{\mathbf{A}}$. We will be interested in the characteristics $t \mapsto (\mathcal{X},\mathcal{V})(t,0,y,w)$, for $|y-(7,0,0)| \leq 1$ and $w \in \R^3_v$ satisfying $w^1 \geq 1$ and $0 \leq w^2 \leq w^3 \leq 1$. We lighten the notation by denoting the trajectory by $(\mathcal{X}_t, \mathcal{V}_t)$ and we introduce $\mathcal{T}_{y,w} \in \R_+ \cup \{ + \infty \}$ the maximal time such that 
\[ \forall \, 0 \leq \tau \leq \mathcal{T}_{y,w}, \qquad \qquad 1 \leq |\mathcal{X}_\tau| - \tau \leq 10, \qquad \frac{\mathcal{X}_\tau^1}{|\mathcal{X}_t|} \geq \frac{1}{4} . \]
In the next two results, we assume that $A$ satisfies the estimate of Lemma \ref{LemdecayAinsta}. We first derive a lower bound for $\mathcal{T}_{y,w}$.
\begin{Lem}\label{LemVprop}
The following properties hold for all $t \in [0,\mathcal{T}_{y,w}]$.
\begin{itemize}
    \item For all $t \in [0,\mathcal{T}_{y,w}]$, we have $\mathcal{V}_{t}^1 \geq w^1$, $(\mathcal{V}^2_t , \mathcal{V}_t^3)=(w^2,w^3)$ and $\frac{\mathcal{X}_\tau^1}{|\mathcal{X}_t|} \geq  1/2$.
    \item We have $\mathcal{T}_{y,w}=+\infty$ or $|\mathcal{X}_{\mathcal{T}_{y,w}}| - \mathcal{T}_{y,w} = 1$. 
    \item We have $\mathcal{T}_{y,w} \geq  \langle w \rangle^2$.
\end{itemize}
\end{Lem}
\begin{proof}
   For the first property, we use \eqref{characInsta} and Lemma \ref{LemdecayAinsta}. It allows to deduce that $\mathcal{X}^1_t \geq y^1+t\widehat{w}^1$, $y^2 \leq \mathcal{X}^2_t \leq y^2+t \widehat{w}^2$ and $y^3 \leq \mathcal{X}_t^3 \leq y^3+ t \widehat{w}^3$. Since $y^1/|y| \geq 1/2$ and $\widehat{w}^1 \geq 1/2$, we obtain $\mathcal{X}_\tau^1/|\mathcal{X}_t| \geq  1/2$. For the second one, we use that $t \mapsto |\mathcal{X}_t|-t$ decreases. Finally, $\mathcal{T}_{y,w} \geq T$, where $T$ satisfies
   \[    |y+T(\widehat{w}^1,0,0)|^2=T^2+2T+1 \; \Longleftrightarrow \; \big( 1- |\widehat{w}^1|^2 \big)T^2+2  \big(1-y^1 \cdot \widehat{w}^1 \big)T+ 1-|y|^2 =0.   \]
   We conclude the proof by observing that, in view of the assumptions on $(y,w)$, we have
   \[  2  \big(y^1 \cdot \widehat{w}^1 -1\big) \geq y^1 \cdot \widehat{w}^1  \geq 3 , \qquad |y|^2-1 \geq 35, \qquad 1-|\widehat{w}^1|^2 \leq 3\langle w \rangle^{-2}.\]
\end{proof}

In the case $\delta =1$, we show a stronger lower bound on the momentum characteristic.

\begin{Lem}\label{LemlowerboundV}
If $\delta =1$, we have $\mathcal{V}^1_t \geq  w^1(1+t)^{c \Lambda}$ for $0 \leq t \leq \mathcal{T}_{y,w}$.
\end{Lem}
\begin{proof}
Using \eqref{characInsta} and Lemma \ref{LemdecayAinsta}, we have
\[ \forall \, t \in [0,\mathcal{T}_{y,w}], \qquad \qquad  \dot{V}^1_t \geq \frac{c \Lambda}{1+t} \big(1+|w^2|+|w^3|+|\mathcal{V}^1_t| \big) . \]
Grönwall's inequality then yields the result.
\end{proof}

\subsection{Proof of Proposition \ref{ProInstaLambdaLarge}}\label{Subsec51}

Let $\varepsilon, \, \Lambda >0$, and let $(f,A) \coloneqq (f^\varepsilon , A^\Lambda)$ be as in the statement of Proposition \ref{ProInstaLambdaLarge}. Assume that $(f,A)$ is a global-in-time solution to \eqref{VWspeci}, with $\delta = 1$, and that there exists a constant $C_{\mathrm{stab}} >0$ such that \eqref{eq:decayinsta} holds. We note in particular that $(y,w)$ is in the support of $f(0,\cdot , \cdot)$. Let us show that $\mathcal{V}_t \to +\infty$ as $t \to +\infty$ if $\Lambda$ is large enough.

\begin{Cor}\label{CorlowerboundV}
In the case $\delta =1$, if $2c\Lambda  \geq 1+ 3/10|w^1|^{-2}$ and if $C_{\mathrm{stab}}\varepsilon$ is small enough, then $\mathcal{T}_{y,w}=+\infty$.
\end{Cor}
\begin{Rq}
    Proceeding similarly, one can show that if $\Lambda$ is small enough, we have $\mathcal{T}_{y,w}+1 \geq  \langle w \rangle^{2/(1-2c \Lambda)} $.
\end{Rq}
\begin{proof}
 Recall that $(\mathcal{V}_t^2,\mathcal{V}_t^3) \equiv (w^2,w^3)$ and $0 \leq w^2, \, w^3 \leq 1$. Hence, using Lemma \ref{LemlowerboundV}, we have
   \[ \dot{\mathcal{X}}_t^1 = \frac{\mathcal{V}_t^1}{\sqrt{1+|w^2|^2+|w^3|^2+|\mathcal{V}_t^1|^2}} \geq  1- \frac{1+|w^2|^2+|w^3|^2}{2|\mathcal{V}_t^1|^2} \geq  1- \frac{3}{2|w^1|^2(1+t)^{2c \Lambda}}   , \]
for all $0 \leq t \leq \mathcal{T}_{y,w}$. We then deduce
\[ \mathcal{X}_t^1 \geq y^1+t-\frac{3}{2|w^1|^2(2c\Lambda -1)}   ,   \]
which implies the result since $|\mathcal{X}_t| \geq \mathcal{X}^1_t$ and $y^1 \geq 6$.
\end{proof}

In view of the support of $f(0,\cdot , \cdot)$, Lemma \ref{LemVprop} and Corollary \ref{CorlowerboundV}, if $f(t,x,v) \neq 0$, then we have $1 \leq |x|-t \leq 10$ and $0 \leq v^2, \, v^3 \leq 1 \leq v^1$, so that $\widehat{v}^1 \geq 1/2$. This allows to conclude the proof of Proposition \ref{ProInstaLambdaLarge}. Indeed, note that by controlling $\partial_{x^1}A$ through Lemma \ref{LemdecayAinsta}, we have
\begin{align*}
    \partial_t \int_{\R^3_x} \int_{\R^3_v} f(t,x,v) \dr v \dr x & = - \int_{\R^3_x} \int_{\R^3_v} \langle v \rangle \, \partial_{x^1} A(t,x) \partial_{v^1} f(t,x,v) \dr v \dr x = \int_{\R^3_x} \int_{\R^3_v} \widehat{v}^1 \partial_{x^1} A(t,x)  f(t,x,v) \dr v \dr x \\
    & \geq \frac{c \Lambda}{1+t} \int_{\R^3_x} \int_{\R^3_v} f(t,x,v) \dr v \dr x.
\end{align*}
Then, the Grönwall inequality finally provides $\| f(t,\cdot , \cdot) \|_{L^1_{x,v}} \geq \| f(0,\cdot , \cdot) \|_{L^1_{x,v}}(1+t)^{c \Lambda}$.

\subsection{Proof of Proposition \ref{ProInstaInidecay}}

Let $\varepsilon, \, \Lambda >0$ and $(f,A)$ be a global-in-time solution to \eqref{VWspeci}, with $\delta =1$, such that $A(0,x)=\partial_t A(0,x)=\Lambda \mathcal{A}(-|x|)/|x|$ and \eqref{eq:decayinsta} holds. We assume further that $C_{\mathrm{stab}}\varepsilon/\Lambda$ is sufficiently small so that the first estimate in Lemma \ref{LemdecayAinsta} holds. As a consequence Lemmata \ref{LemVprop}--\ref{LemlowerboundV} are satisfied. 

Let us show that the original decay space for $f$ is not asymptotically stable for a large class of natural decay rates, thereby demonstrating that some loss of decay is necessary. Consider now a weight function $W \in C^1 (\R_+, [1,+\infty[)$ such that $W(s) \to +\infty$ as $s \to + \infty$ and $sW'(s)/W(s) \geq b$, for some constant $b >0$. Then, we have
\[ \T_{\mathbf{A}} \big( W |f| \big) =   \partial_{x^1} A v^1 W'(\langle v \rangle) |f|.  \]
We consider a characteristic $t \mapsto (\mathcal{X}_t,\mathcal{V}_t)$ satisfying the same assumptions as those imposed in Section \ref{Subsec51}. As $2V^1_t \geq \langle \mathcal{V}_t \rangle$ on $[0,\mathcal{T}_{y,w}]$, we then obtain from Lemma \ref{LemdecayAinsta}, that
\[ \forall \, t \in \big[ 0 ,\mathcal{T}_{y,w} \big], \qquad \qquad \T_{\mathbf{A}} \big( W |f| \big) (t, \mathcal{X}_t, \mathcal{V}_t) \geq \frac{c \Lambda}{1+t} \frac{ \langle \mathcal{V}_t \rangle W'(\langle v \rangle)}{W(\langle v \rangle)} W(\langle \mathcal{V}_t \rangle) \cdot \big|f (t,\mathcal{X}_t,\mathcal{V}_t ) \big|.  \]
The Grönwall inequality and the assumption on $W$ provides
\[  \forall \, t \in \big[ 0 , \mathcal{T}_{y,w} \big], \qquad W \big( \langle \mathcal{V}_t \rangle \big) \big| f (t,\mathcal{X}_t,\mathcal{V}_t )  \big| \geq  W ( \langle w \rangle) \big|f(0,y,w) \big| (1+t)^{bc \Lambda } . \]
According to Lemma \ref{LemVprop}, we then obtain, for all $\tau=w^1 \geq 1$ and $0 \leq w^2, \, w^3 \leq 1$,  
\[ W \big( \langle \mathcal{V}_\tau \rangle \big) \big| f (\tau,\mathcal{X}_\tau,\mathcal{V}_\tau ) \big| \geq \langle \tau \rangle^{2bc \Lambda } W ( \langle \tau \rangle) \big|f(0,y,\tau,w^2,w^3) \big| . \]

\renewcommand{\refname}{References}
\bibliographystyle{abbrv}
\bibliography{biblio}

\end{document}